\documentclass[12pt]{amsart}
\usepackage{a4wide,enumerate,xcolor}
\usepackage{amsmath,amssymb}
\usepackage{comment}

\usepackage{graphicx}
\usepackage{float}
\usepackage{multirow}
\usepackage{multicol}

\usepackage{algorithm}
\usepackage{algorithmic}

\usepackage{diagbox}

\let\pa\partial

\def\ii{\mathrm{i}}

\numberwithin{equation}{section}

\newtheorem{theorem}{Theorem}[section]
\newtheorem{lemma}[theorem]{Lemma}
\newtheorem{proposition}[theorem]{Proposition}
\newtheorem{remark}[theorem]{Remark}
\newtheorem{corollary}[theorem]{Corollary}

\graphicspath{{figures/}}

\begin{document}

\title[A Bernoulli Phase-Fitted Method for the Helmholtz Equation]{A Bernoulli Phase-Fitted finite difference method \\ with wavenumber-explicit analysis\\
for the Helmholtz problem}

\author[A. J\"ungel]{Ansgar J\"ungel}
\address{Institute of Analysis and Scientific Computing, TU Wien, Wiedner Hauptstra\ss e 8--10, 1040 Wien, Austria}
\email{juengel@tuwien.ac.at} 

\author[P. Li]{Panchi Li}
\address{School of Mathematical Sciences, Soochow University, Suzhou, 215006, China}
\email{lipch@suda.edu.cn} 

\author[Z. Sun]{Zhiwei Sun}
\address{Institute of Analysis and Scientific Computing, TU Wien, Wiedner Hauptstra\ss e 8--10, 1040 Wien, Austria}
\email{zhiwei.sun@tuwien.ac.at}

\author[Z. Zhang]{Zhiwen Zhang}
\address{Department of Mathematics, The University of Hong Kong, Hong Kong, China}
\email{zhangzw@hku.hk} 

\date{\today}

\thanks{The first author acknowledges partial support from the Austrian Science Fund (FWF), grant 10.55776/PAT2687825, and from the Austrian Federal Ministry for Women, Science and Research and implemented by \"OAD, grant MULT09/2025. This work has received funding from the European Research Council (ERC) under the European Union's Horizon 2020 research and innovation programme, ERC Advanced Grant NEUROMORPH, no.~101018153. The research of Z. Zhang was supported by the National Natural Science Foundation of China (Project 92470103), the Hong Kong RGC grant (Projects 17304324 and 17300325), the Seed Funding Programme for Basic Research (HKU), and the Hong Kong RGC Research Fellow Scheme 2025. For open-access purposes, the authors have applied a CC BY public copyright license to any author-accepted manuscript version arising from this submission.}

\begin{abstract}
A new Bernoulli phase-fitted finite difference method for the Helmholtz equation is introduced, obtained by applying a complexified Scharfetter--Gummel flux to the one-way factors of the operator. The rigorous analysis is developed for the one-dimensional Helmholtz problem with impedance boundary conditions. For the homogeneous problem, the scheme reproduces sampled plane-waves exactly, both in the interior and at the discrete impedance boundary closures. For the inhomogeneous problem, we prove wavenumber-explicit stability, consistency, and second-order convergence estimates for all nondegenerate mesh wavenumbers \(kh\notin\pi\mathbb Z\). Under the fixed-resolution condition \(kh\le s_0<\pi\) and \(kL\ge\pi\), the estimates yield a pollution-free convergence theory. Numerical experiments confirm the plane-wave exactness and the predicted convergence behavior, and show favorable fixed-resolution performance compared with standard and dispersion-corrected finite difference methods.
\end{abstract}

\keywords{Bernoulli Phase-Fitted method, Helmholtz problem, impedance boundary conditions, Scharfetter--Gummel discretization, wavenumber-explicit analysis, pollution-free convergence.}

\subjclass[2020]{35J05, 65N06, 65N12, 65N15.}

\maketitle

\section{Introduction}
The Helmholtz equation is a fundamental model for time-harmonic wave
propagation in acoustics, electromagnetics, and seismic imaging.
In the high-frequency regime, characterized by large wavenumbers $k$,
its solutions are highly oscillatory, and their accurate numerical
approximation remains a central challenge in scientific computing.
A well-known difficulty is the pollution effect \cite{babuska1997pollution,bayliss1985accuracy,DispersionPollution1999Deraemaeker,IHLENBURG19959,ihlenburg1997Finite}: 
For standard discretization methods, 
as the wavenumber $k$ increases, the numerical error cannot in general
be controlled solely by maintaining a fixed number of degrees of
freedom per wavelength. Consequently, to achieve a prescribed accuracy,
the total number of degrees of freedom must grow faster than the
natural scaling $O(k^d)$, where $d$ is the spatial dimension.

A large body of work has therefore been devoted to mitigating the
pollution effect in Helmholtz discretizations. In the context of finite element methods, high-order and $hp$-FEM techniques can substantially reduce pollution when the polynomial degree is increased together with mesh refinement
\cite{LAFONTAINE202259,melenk2010convergence,doi:10.1137/090776202}.
A particularly relevant class is Trefftz and wave-based methods, where the approximation spaces are built from local solutions of the
underlying differential equation and hence incorporate oscillatory
behavior directly
\cite{cessenat1998Application,hiptmair2011Plane,hiptmair2014Trefftz,hiptmair2016survey}.
These developments suggest that incorporating the local oscillatory structure of the Helmholtz operator can be important for high-frequency accuracy.


Within the finite difference framework, the same concern is often formulated in terms of numerical dispersion. Many approaches reduce the dispersion error by modifying the discrete operator, including optimized compact schemes \cite{CHEN20128152,CHENG20172345}, dispersion-minimizing schemes \cite{DASTOUR2021113544,Dastour2021,ref1,WU20182520}, shifted-wavenumber discretizations designed to match the exact phase velocity more accurately \cite{cocquet2017finite,cocquet2021closed}, and asymptotic dispersion corrections for general finite difference schemes \cite{cocquet2024asymptotic}. In one dimension, pollution-free finite difference constructions have also been studied in related settings \cite{wong2011exact,wang2014pollution}.
In this paper, we follow a different route: Rather than starting from a second-order stencil and then tuning its phase behavior, we build the discretization from the one-way propagation structure of the Helmholtz operator. In this way, the phase fitting is imposed on the one-way factors and leads directly to exactness for homogeneous plane-waves at the discrete level.

\subsection{From one-way factorization to phase-fitted fluxes}

We first describe the structural idea behind the method. 
The starting point is the observation that the
one-dimensional Helmholtz operator admits the factorization
\begin{align}\label{1.helm}
\partial_{xx}+k^2=(\partial_x+\ii k)(\partial_x-\ii k),
\end{align}
which separates the two one-way propagation components of the wave. In particular,
\begin{equation*}
   (\partial_{x} + \ii k) e^{-\ii kx}=0, \qquad (\partial_{x} - \ii k) e^{\ii kx}=0.
\end{equation*}
Thus the first-order operators annihilate outgoing and incoming plane-waves,
respectively. 
This suggests a different discretization strategy: rather than approximating the Helmholtz operator directly, one may first discretize the one-way factors in a way that preserves their plane-wave annihilation property.

To discretize these one-way operators, we draw inspiration from the
Scharfetter--Gummel (SG) discretization \cite{ScGu69}, which was
originally developed for semiconductor drift--diffusion models.
In one space dimension, a drift--diffusion equation can be written in the flux form
\begin{equation}\label{eqn:def-flux}
    \partial_t u = \partial_x F,
    \qquad F= \partial_x u -  v u = (\partial_x - v)u,
\end{equation}
where $v$ denotes the drift velocity.
The SG scheme approximates the flux at cell interfaces by locally solving the stationary flux equation, which leads
to 
\begin{equation}\label{SG in drift}
    F_{i+1/2}= \frac{1}{h}
    \big( B(h v) u_{i+1} - B(-h v) u_i\big).
\end{equation}
Here, $B(s)$ for $s\in\mathbb R$ is the Bernoulli function defined by
$B(s)=s/(e^s-1)$ for $s\neq 0$ and $B(0)=1$. 
This one-dimensional SG flux is local and edge-based. Hence, it is naturally compatible with multidimensional finite volume methods \cite{chainais-hillairet2003Finite,Bes12}, and incorporated into finite element frameworks through edge-averaged constructions
\cite{xu1999Monotone,lazarov2012exponential}. 

In this paper, we transfer this SG-flux idea from particle transport to wave propagation.  We replace the real drift $v$ in \eqref{eqn:def-flux} by the imaginary wavenumber $\pm \ii k$. Formally, the flux fitting \eqref{SG in drift} for $\partial_x - v$ becomes a phase fitting for the one-way factors $\partial_x\pm\ii k$ in Helmholtz operators.
In this way, we develop a new Bernoulli phase-fitted (BPF) finite difference method for the Helmholtz equation. The method preserves the one-way factorization structure
and local exactness for plane-waves.

This viewpoint follows the philosophy of the classical SG method: First, a one-dimension\-al fitted flux is derived, providing a natural building block for more general discretizations. The present paper proceeds in the same spirit and develops the one-dimensional Helmholtz analogue first. Extensions to multidimensional finite volume and finite element discretizations are natural. However, anisotropic numerical dispersion and boundary effects for Helmholtz problems become essential in several space dimensions. These questions are left
for future work. 
A central question, already in this one-dimensional setting, is whether this local exactness yields a reduction of the dispersion pollution in inhomogeneous Helmholtz problems. 
To answer this question, we prove wavenumber-explicit stability and pollution-free convergence estimates.

\subsection{Main results and contributions}

To make the factorized BPF construction fully transparent and to obtain
wavenumber-explicit estimates, we focus on the one-dimensional Helmholtz problem 
\begin{equation}\label{eq:Helmholtz}
u_{xx}(x) + k^2\,u(x) = f(x), \qquad x\in(0,L),
\end{equation}
with the impedance boundary conditions
\begin{equation}\label{eq:bc}
u_x(0) - \ii k\,u(0) = g_0,\qquad
u_x(L) + \ii k\,u(L) = g_L,
\end{equation}
where $g_0,g_L\in\mathbb C$ prescribe the incoming wave data at the two
boundaries.

For the homogeneous case $f(x)=0$, we show that the
proposed BPF scheme reproduces the
sampled plane-wave solutions exactly, both in the interior and at the
discrete boundary closures.
For the inhomogeneous problem, we first use the factorized structure of the scheme to derive stability bounds for all
$kh\notin \pi\mathbb Z$. We then exploit exact plane-wave lifting and Fourier-symbol estimates to control the consistency error. 
Combining these ingredients yields a wavenumber-explicit convergence analysis.

More precisely, the convergence estimate for the grid error $e$ takes the form of
\[
k\|e\|_{0,h}+|e|_{1,h}\le C(kh,kL)\, h^2,
\qquad kh\notin \pi\mathbb Z
\]
with the constant $C(kh,kL)$ given explicitly in terms of $kh$ and $kL$. 
In particular, under the conditions $kh\le s_0<\pi$ for some $s_0>0$ and $kL\ge\pi$,
the constant remains uniformly bounded with respect to the
wavenumber. Thus, the method admits a pollution-free $O(h^2)$ convergence estimate under fixed resolution. The advantage becomes clearer when we compare the wavenumber scaling: The dispersion-corrected scheme has an $O(k^2h^2)$ error \cite{cocquet2024asymptotic}, while the classical centered finite difference scheme introduces an $O(k^3h^2)$ error \cite{gander2025fourier}. When $kh$ is fixed, the latter two correspond to an $O(1)$ and an $O(k)$ pollution contribution. By contrast, the present estimate gives an $O(k^{-2})$ decay. This interesting behavior is also reflected in the numerical experiments.


The main contributions of this paper are as follows:
\begin{itemize}
\item We propose a BPF discretization for the Helmholtz equation, obtained by applying a complexified SG flux to the one-way factors of the operator. The scheme preserves the one-way factorization structure
and is locally exact on plane-waves.

\item For the one-dimensional impedance problem, we show that the BPF discretization reproduces homogeneous plane-waves exactly at the grid level, yielding an exact discrete impedance closure (Proposition \ref{lem:exact-bdry}).

\item We develop a wavenumber-explicit analysis for the one-dimensional inhomogeneous impedance problem, leading to stability, consistency, and pollution-free second-order convergence estimates (Theorems \ref{thm:main-stability}, \ref{thm:residual-bounds}, and \ref{thm:main-convergence}).

\item Numerical experiments confirm the plane-wave exactness and the predicted convergence behavior and show favorable performance compared with standard and dispersion-corrected finite difference schemes. 
A two-dimensional aligned plane-wave test further illustrates the directional
exactness, discussed in Remark~\ref{remark:directional-multidimension}.
\end{itemize}

\subsection{Organization of the paper}
The remainder of the paper is organized as follows.
In Section \ref{sec:model}, we derive the BPF discretization and prove its plane-wave exactness.  Section~\ref{sec:energy-error} is devoted to the well-posedness and
stability analysis of the discrete scheme. 
In Section~\ref{sec:consistency}, we establish consistency and convergence estimates.
Finally, numerical experiments are presented in Section \ref{sec:numerical-experiments}.


\section{Bernoulli phase-fitted discretization}\label{sec:model}

In this section, we derive the BPF finite
difference discretization for the Helmholtz equation, and prove its local exactness for homogeneous plane-waves at the discrete level.

\subsection{Grid notation}

We introduce the grid notation used for the one-dimensional impedance problem \eqref{eq:Helmholtz}--\eqref{eq:bc}. The interval $[0,L]$ is divided into $n$ uniform subintervals of length $h:=L/n$ with grid points $x_i:=ih$ for $i=0,1,\dots,n$.
For grid functions $v=\{v_i\}_{i=0}^n$ and $w=\{w_i\}_{i=0}^n$, we define
the discrete inner product and $L^2$ norm on the interior grid by
\[
(v,w)_h := h\sum_{i=1}^{n-1} v_i\overline{w_i},\qquad
\|v\|_{0,h}^2 := (v,v)_h,
\]
where $\overline{w_i}$ denotes complex conjugation.
The forward difference operator and the associated discrete $H^1$ seminorm are defined as
\[
(\nabla_h v)_i := \frac{v_{i+1}-v_i}{h},\qquad i=0,\dots,n-1,
\qquad
|v|_{1,h}^2 := h\sum_{i=0}^{n-1} |(\nabla_h v)_i|^2.
\]
Finally, the standard three-point approximation to the second derivative equals
\[
(\Delta_h v)_i := \frac{v_{i+1}-2v_i+v_{i-1}}{h^2},\qquad i=1,\dots,n-1.
\]

\subsection{Complexified Scharfetter--Gummel operators}\label{subsec:cSG}

To discretize the one-way operators in the Helmholtz
factorization \eqref{1.helm}, we apply a complexified
SG flux. We first extend the Bernoulli function at complex arguments $z\in\mathbb C$ by
\begin{equation*}
   B(z) = \frac{z}{e^z - 1} \quad\text{for $z \neq 0$}, \qquad\text{and}\qquad
       B(0) = 1.
\end{equation*}
The following useful properties are used later:
\begin{equation}\label{eq:basicIds}
B(-z)=e^{z}B(z),\qquad B(-z)-B(z) = z.
\end{equation}
We then
complexify the SG flux by replacing the real drift velocity $v$
in \eqref{SG in drift} with the imaginary wavenumber $\pm\ii k$.
This leads to the discretization of one-way operators $\partial_x\pm\ii k$:
\begin{equation*}
\begin{aligned}
    D_k^+u_i &:= \frac{1}{h}\!\Big(B(\ii kh)\,u_{i+1}-B(-\ii kh)\,u_i\Big), \\
    D_k^-u_i &:= \frac{1}{h}\!\Big(B(-\ii kh)\,u_{i}-B(\ii kh)\,u_{i-1}\Big).
\end{aligned}
\end{equation*}
These discrete operators annihilate the corresponding one-way
plane-waves exactly. Indeed, by \eqref{eq:basicIds} we have
\begin{equation}\label{eq:one-way-exactness}
   D_k^+e^{\ii kx_i}=0,\qquad D_k^-e^{-\ii kx_i}=0.
\end{equation}

\subsection{BPF finite difference scheme}\label{subsec:FD}

We now compose the one-way operators $D_k^+$ and $D_k^-$ introduced above.
Applying these operators to the Helmholtz equation
\eqref{eq:Helmholtz} together with the boundary conditions \eqref{eq:bc} leads to the discrete system
\begin{equation}\label{discretization SG}
    \left\{\begin{aligned}
    D_k^- D_k^+ u_i = f_i&, \qquad i=1,\dots,n-1,\\
    D_k^+u_0 = m(kh)\,g_0&, \qquad
    D_k^-u_n =m(kh)\, g_L,
    \end{aligned}\right.
\end{equation}
where $f_i := f(x_i)$. The complex-valued correction factor
\begin{equation}\label{def.m}
    m(s) =e^{-\ii s/2}\cos(s/2)
\end{equation}
is chosen in such a way that the discrete boundary operators reproduce
the exact plane-wave solutions; see
Proposition~\ref{lem:exact-bdry}.
Note that $m(s)\neq0$ for $s\in(0,\pi)$.

We further introduce the phase-fitted weight
\begin{equation*}
\Theta(s):=|B(\ii s)|^2 = B(\ii s)B(-\ii s),
\qquad s\in\mathbb{R},
\end{equation*}
which has the properties
\[
\Theta(s)= \frac{s^2}{4\sin^2(s/2)},
\quad s\notin 2 \pi\mathbb Z,
\quad\text{and}\quad 
1\le \Theta(s) \le \frac{\pi^2}{4},
\quad 0\le s\le \pi.
\]
A direct algebraic calculation (see
Proposition~\ref{prop:factorization} in Appendix~A)
shows that the composition $D_k^-D_k^+$ in \eqref{discretization SG} can be written in the three-point form
\begin{equation}\label{eq:exponential-fitted}
\Theta(kh)\,(\Delta_h u)_i + k^2 u_i = f_i.
\end{equation}
This representation shows that the BPF discretization
is related to a three-point finite difference scheme with a shifted wavenumber $\widehat k = k/\sqrt{\Theta(kh)}$; see \cite{ernst2013multigrid,wong2011exact}, together with a scaled source term $f_i/\Theta(kh)$.
This source scaling is a structural consequence of the factorized BPF construction, and is absent in the corresponding shifted-wavenumber or dispersion-corrected discretization \cite{cocquet2024asymptotic}. For inhomogeneous problems, it removes the leading $O(k^2h^2)$ consistency contribution and leads to the $O(h^2)$ error estimate proved in this work.

\subsection{Exactness for plane-waves}\label{subsec:symbol}

Now we prove that the BPF discretization reproduces the exact plane-wave solutions of the homogeneous Helmholtz equation. Indeed, consider the homogeneous problem ($f\equiv 0$).
Its general solution on $(0,L)$ is
\begin{equation}\label{homogeneous interior}
u(x)=\alpha\,e^{\ii k x}+\beta\,e^{-\ii k x},\qquad \alpha,\beta\in\mathbb C.
\end{equation}
The coefficients $\alpha$ and $\beta$ are determined by the impedance boundary conditions \eqref{eq:bc}:
\begin{equation}\label{eq:alpha-beta}
    -2\ii k\,\beta = g_0,
    \qquad
    2\ii k\,e^{\ii kL}\alpha = g_L.
\end{equation}

\begin{proposition}[Exactness for plane-waves]\label{lem:exact-bdry}
Let $f\equiv0$ and consider the con\-tin\-uous solution
\eqref{homogeneous interior}.
Define the sampled sequence $u_i:=u(x_i)$.
Then it solves
\begin{equation*}
D_k^- D_k^+u_i=0, \quad
(D_k^+u)_0=m(kh)g_0,\quad (D_k^-u)_n=m(kh)g_L,
\end{equation*}
where $i=1,\dots,n-1$, recalling definition \eqref{def.m} of $m(s)$.
\end{proposition}

\begin{proof}
We recall that the discrete one-way operators annihilate
the corresponding plane-waves; see \eqref{eq:one-way-exactness}. Hence,  for the general solution \eqref{homogeneous interior},
linearity yields
\[
(D_k^-D_k^+u)_i
=
\alpha D_k^-D_k^+e^{\ii kx}
+
\beta D_k^-D_k^+e^{-\ii kx}
=0 ,
\]
which proves the interior equation. We infer from \eqref{eq:rewrite D+} and \eqref{eq:key-id} that
\[
(D_k^+ e^{-\ii kx})_0
= 
\frac{B(\ii kh)}{h}\big(e^{-\ii kh} - e^{\ii kh}\big)
=
-\frac{2\ii}{h}B(\ii kh)\sin(kh)
=
-2\ii k\,m(kh),
\]
and similarly $(D_k^- e^{\ii kx})_n = 2\ii k\,m(kh)e^{\ii kL}$. We substitute \eqref{eq:alpha-beta} to obtain $(D_k^+u)_0=m(kh)g_0$ and $(D_k^-u)_n=m(kh)g_L$, finishing the proof.
\end{proof}

As a consequence, the BPF discretization introduces no numerical dispersion for homogeneous plane-waves in the interior. Moreover, in the impedance setting, the boundary closure produces no artificial reflection.

\begin{remark}[Dirichlet data]
For Dirichlet boundary conditions, the corresponding exactness is immediate.
Indeed, if \(u(x)=\alpha e^{\ii kx}+\beta e^{-\ii kx}\) and the boundary values are
prescribed by its trace, then the sampled values \(u_i=u(x_i)\) satisfy
\[
D_k^-D_k^+u_i=0,\qquad i=1,\ldots,n-1,
\]
together with the exact boundary data. 
\end{remark}

\begin{remark}[Directional exactness in several dimensions]
\label{remark:directional-multidimension}
The exactness mechanism above is local and factorized, and it has a
directional analogue in several space dimensions. Define
\(\mathbf{k}=(k_1,\ldots,k_d)\in\mathbb R^d\) with
\(|\mathbf{k}|=k\). The Helmholtz operator admits
the directional factorization
\[
\Delta+k^2=(\nabla+\ii\mathbf{k})\cdot(\nabla-\ii\mathbf{k}).
\]
It suggests applying the one-dimensional complexified SG flux in each
coordinate direction with the corresponding component \(k_j\). This gives the directional BPF interior operator
\[
\mathcal L_h^{\mathbf{k}}u
=
\sum_{j=1}^d D_{j,k_j}^-D_{j,k_j}^+u,
\]
where \(D_{j,k_j}^{\pm}\) denotes the one-dimensional BPF one-way
operator in the \(x_j\)-direction. It follows from Proposition~\ref{lem:exact-bdry}, applied in each direction, that the sampled aligned plane-wave \(u(\mathbf{x})=e^{\ii\mathbf{k}\cdot\mathbf{x}}\)
satisfies $\mathcal L_h^{\mathbf{k}}u=0$ at all interior grid points. Thus, the directional extension preserves exactness
for aligned plane-waves.
This indicates that the BPF-type construction is local and directional; a full multidimensional discretization
and analysis are left for future work.
\end{remark}

\section{Stability and well-posedness}\label{sec:energy-error}

We establish the main stability result for the BPF scheme and prove the well-posedness of the discrete problem.
Throughout this section, we assume that $kh\notin \pi\mathbb Z$,
so that the Nyquist degeneracy is avoided and
for brevity, we write $\Theta:=\Theta(kh) = |B(\ii kh)|^2$.

Our goal is to derive estimates that are explicit with respect to the
wavenumber and remain uniformly bounded under a fixed resolution
constraint.

\begin{theorem}[Well-posedness and $k$-explicit stability]
\label{thm:main-stability}
Let $u_h=\{u_i\}_{i=0}^n$ denote the solution to the BPF scheme
\eqref{discretization SG}.
For $s:=kh$ and $ t:=kL$,
define
\begin{equation}\label{eq:def-const-A0}
A_0(s,t)
:= \frac{L}{\sqrt{2\Theta(s)}}\,\sec\!\Big(\frac{s}{2}\Big)
   +\frac{L}{2t}\,\sec^2\!\Big(\frac{s}{2}\Big),
\end{equation}
Then, for any $kh\notin \pi\mathbb Z$, and for any source $f$ and boundary data $g_0,g_L\in\mathbb C$, the BPF
scheme admits a unique discrete solution satisfying
\begin{align}
k\|u_h\|_{0,h}
&\le
A_0(kh,kL)\,\|f\|_{0,h}
+\frac{\sqrt L}{2}\bigl(|g_0|+|g_L|\bigr),
\label{eq:main-L2-general}
\\
\sqrt{\Theta}\,|u_h|_{1,h}
&\le
A_0(kh,kL)\,\|f\|_{0,h}
+\frac{\sqrt L}{2}\bigl(|g_0|+|g_L|\bigr).
\label{eq:main-H1-general}
\end{align}
In particular, for any fixed $s_0<\pi$, the quantities
$A_0(kh,kL)$ remains uniformly bounded with respect to
$k$ whenever $kh\le s_0$ and $kL\ge \pi$.
\end{theorem}

The condition $kh\le s_0<\pi$ corresponds to the principal Nyquist regime, which admits a number of more than two grid points per wavelength (PPW),
while $kL\ge\pi$ means that the interval $(0,L)$ contains at least half
a wavelength. 

The proof of Theorem~\ref{thm:main-stability} is based on two key
ingredients.
First, we derive stability estimates for the discrete problem with
homogeneous discrete radiation conditions.
Second, we exploit the exactness result of
Proposition~\ref{lem:exact-bdry}, which shows that the BPF scheme
reproduces the homogeneous plane-wave component exactly.
This allows us to treat general boundary data by an exact lifting
argument.

\subsection{Reduction to homogeneous discrete radiation conditions}

We first consider the BPF discretization with homogeneous discrete
radiation conditions,
\begin{equation}\label{eq:bpf-hom-bc}
\left\{
\begin{aligned}
D_k^-D_k^+u_i &= f_i, \qquad i=1,\dots,n-1,\\
(D_k^+u)_0 &= 0,\qquad
(D_k^-u)_n = 0.
\end{aligned}
\right.
\end{equation}
This reduced problem contains the essential stability mechanism of the
scheme. We isolate \eqref{eq:bpf-hom-bc} since the contribution
of the boundary data can be represented exactly by plane-waves.
Indeed, Proposition~\ref{lem:exact-bdry} shows that the homogeneous
Helmholtz equation with impedance boundary data is reproduced exactly
by the discrete BPF scheme at the nodal level.
Consequently, the full discrete problem with general boundary data can
be reduced to \eqref{eq:bpf-hom-bc} by subtracting an exact plane-wave
lifting.

We therefore first establish $k$-explicit stability estimates for
\eqref{eq:bpf-hom-bc}. The lifting argument and the proof of the full
stability theorem will be given in
Section~\ref{sec:energy-error-general}.

\subsection{Basic flux and energy identities}

We derive the basic estimates for the homogeneous-boundary problem
\eqref{eq:bpf-hom-bc}.
Since the one-way fluxes $D_k^+u$ and $D_k^-u$ are naturally defined on
staggered index sets, we slightly abuse the notation and write
\[
\|D_k^+u\|_{0,h}^2 := h\sum_{i=0}^{n-1}|D_k^+u_i|^2,
\qquad
\|D_k^-u\|_{0,h}^2 := h\sum_{i=1}^{n}|D_k^-u_i|^2.
\]

We start with $L^2$ bounds for the discrete one-way fluxes associated
with the homogeneous discrete radiation problem. 

\begin{lemma}[Flux estimate]
\label{lem:flux-estimate}
Let $u=\{u_i\}_{i=0}^n$ solve \eqref{eq:bpf-hom-bc}. Then
\begin{equation}
\label{eq:flux-estimate-both}
\|D_k^+u\|_{0,h}^2+\|D_k^-u\|_{0,h}^2
\le
\frac{L^2}{\Theta}\,\|f\|_{0,h}^2.
\end{equation}
\end{lemma}

\begin{proof}
Set $w_i:=(D_k^+u)_i$ for $i=0,\dots,n-1$.
We deduce from $D_k^-D_k^+u=f$ and the definition of $D_k^-$ that
\[
B(-\ii kh)\,w_i-B(\ii kh)\,w_{i-1}=h f_i,
\qquad i=1,\dots,n-1.
\]
Using the identity $B(\ii kh)=e^{-\ii kh}B(-\ii kh)$ yields the
recurrence
\[
w_i=e^{-\ii kh}w_{i-1}+\frac{h}{B(-\ii kh)}\,f_i,
\qquad i=1,\dots,n-1.
\]
Since $(D_k^+u)_0=0$, an iteration gives
\begin{equation*}
w_i=\frac{h}{B(-\ii kh)}\sum_{j=1}^i e^{-\ii kh(i-j)}\,f_j,
\qquad i=1,\dots,n-1.
\end{equation*}
Taking absolute values and using $|e^{-\ii kh(i-j)}|=1$ and
$|B(-\ii s)|^2=\Theta(s)$, we find that
\begin{align*}
|w_i|
&\le \frac{h}{\sqrt{\Theta}}\sum_{j=1}^i |f_j|
\le \frac{h}{\sqrt{\Theta}}\sqrt{i}
   \bigg(\sum_{j=1}^i |f_j|^2\bigg)^{1/2} 
\le \frac{\sqrt{ih}}{\sqrt{\Theta}}\,\|f\|_{0,h}.
\end{align*}
It follows from $\sum_{i=0}^{n-1} i \le n^2/2 $ and $nh=L$ that
\begin{equation}\label{eq:Dplus-L2-bound}
\|D_k^+u\|_{0,h}=\|w\|_{0,h}
\le \frac{L}{\sqrt{2\Theta}}\,\|f\|_{0,h}.
\end{equation}
A completely analogous backward recursion, starting from
$(D_k^-u)_n=0$ leads to
\begin{equation}\label{eq:Dminus-L2-bound}
\|D_k^-u\|_{0,h}
\le \frac{L}{\sqrt{\Theta}}\,\|f\|_{0,h}.
\end{equation}
Combining \eqref{eq:Dplus-L2-bound} and \eqref{eq:Dminus-L2-bound} gives \eqref{eq:flux-estimate-both}.
\end{proof}


The next result relates the discrete one-way fluxes to the standard
discrete $H^1$ norm and the boundary data.

\begin{lemma}[Flux--energy relation]\label{lem:flux-energy}
For any $u=\{u_i\}_{i=0}^n$ and all $kh\notin\pi\mathbb Z$,
\begin{equation}\label{eq:flux-energy}
\|D_k^+u\|_{0,h}^2+\|D_k^-u\|_{0,h}^2
=
2\Theta\cos(kh)\,|u|_{1,h}^2
+2k^2\,\|u\|_{0,h}^2 + k^2h\big(|u_{n}|^2 + |u_0|^2\big).
\end{equation}
\end{lemma}

\begin{proof}
Using the identity $B(-z)-B(z)=z$, we rewrite
\begin{align}\label{eqn:rewrite D_k+}
    (D_k^+u)_i
    &= \frac{B(\ii kh)\,u_{i+1} - B(-\ii kh)\,u_i}{h}
    = B(\ii kh)\nabla_h u_i - \ii k\, u_i, \\
    \label{eqn:rewrite D_k-}
    (D_k^-u)_{i+1}
    &= \frac{B(-\ii kh)\,u_{i+1} - B(\ii kh)\,u_i}{h}
    = B(\ii kh)\nabla_h u_i + \ii k\, u_{i+1}.
\end{align}
Hence, it follows that
\begin{align*}
    |D_k^+u_i|^2 + |D_k^-u_{i+1}|^2
    &= 2|B(\ii kh)|^2\,|\nabla_h u_i|^2
       + k^2\,\big(|u_{i+1}|^2+|u_i|^2\big)  + T_{\rm cross},
\end{align*}
where the cross terms are collected into 
\begin{align*}
    T_{\rm cross}
    &=
    2\Re\Big(
B(\ii kh)\,\nabla_h u_i \cdot \overline{\ii kh(\nabla_h u_i)}
\Big)
=
\ii kh\Big(
B(-\ii kh) -B(\ii kh)
\Big) |\nabla_h u_i|^2
\end{align*}
Using $B(-z)-B(z)=z$ again,
this gives
\begin{equation*}
|D_k^+u_i|^2+|D_k^-u_{i+1}|^2
=
\big(2|B(\ii kh)|^2-k^2h^2\big)|\nabla_h u_i|^2
+ k^2\big(|u_{i+1}|^2 + |u_i|^2\big).
\end{equation*}
Since $|B(\ii s)|^2=\Theta(s)$ and
\[
2\Theta(s)-s^2
=
\frac{s^2}{2\sin^2(s/2)}-s^2
=
\frac{s^2\cos s}{2\sin^2(s/2)}
=
2\Theta(s)\cos s,
\]
we obtain
\[
|D_k^+u_i|^2+|D_k^-u_{i+1}|^2
=
2\Theta(kh)\cos(kh)\,|\nabla_h u_i|^2
+k^2\big(|u_{i+1}|^2 + |u_i|^2\big).
\]
Summing over $i=0,\dots,n-1$ and multiplying by $h$ yields
\eqref{eq:flux-energy}.
\end{proof}

As a byproduct of Lemmas~\ref{lem:flux-estimate}
and~\ref{lem:flux-energy}, we have the following estimate.
\begin{corollary}[Auxiliary energy bound]
\label{rem:low-frequency-coercive}
Let $u$ solve \eqref{eq:bpf-hom-bc}. For all $kh\notin\pi\mathbb Z$,
\begin{equation}
\label{eq:low-frequency-coercive}
\Theta\cos(kh)\,|u|_{1,h}^2
+
k^2\|u\|_{0,h}^2
+
\frac{k^2h}{2}\bigl(|u_0|^2+|u_n|^2\bigr)
\le
\frac{L^2}{2\Theta}\,\|f\|_{0,h}^2.
\end{equation}
\end{corollary}

\begin{remark}[Coercivity in the low-frequency regime]
Although \eqref{eq:low-frequency-coercive} is valid for all
$kh\notin \pi\mathbb Z$, it is coercive in the principal Nyquist regime $0<kh<\pi$ only for $0<kh<\pi/2$, since $\cos(kh)$ changes sign at
$kh=\pi/2$. Hence, the estimate
\eqref{eq:low-frequency-coercive} alone is not sufficient to establish $H^1$-control over the entire interval $0<kh<\pi$, and an additional discrete energy identity, proved in the next subsection, is needed. 
\end{remark}


\subsection{Uniform stability for homogeneous discrete radiation conditions}

We now establish $k$-explicit stability estimates for the homogeneous
discrete radiation problem \eqref{eq:bpf-hom-bc} that remain valid for all
$kh\notin \pi\mathbb Z$, in particular for the whole principal Nyquist regime $0<kh<\pi$.
The key point is the discrete energy identity associated with the three-point formulation of the BPF scheme.

\begin{lemma}[Discrete energy identity]
\label{lem:complex-energy}
Let $u=\{u_i\}_{i=0}^n$ solve \eqref{eq:bpf-hom-bc}. Then
\begin{align}
\Theta
|u|_{1,h}^2
-\frac{k^2h}{2}\bigl(|u_0|^2+|u_n|^2\bigr)
-
k^2\|u\|_{0,h}^2
&=
\Re(f,u)_{h}.
\label{eq:real-part-energy}
\end{align}
\end{lemma}

\begin{proof}
According to the representation \eqref{eq:exponential-fitted}, we
write the interior equation as
\[
\Theta(kh)\,(\Delta_h u)_i+k^2u_i=f_i,
\qquad i=1,\dots, n-1.
\]
Multiplying by $h\,\overline{u_i}$, summing over $i=1,\dots, n-1$,
and using the standard summation-by-parts identity
\[
h\sum_{i=1}^{n-1} (\Delta_h u)_i\,\overline{u_i}
=
-|u|_{1,h}^2
+\big(
\nabla_h u_{n-1}\,\overline{u_n}
-\nabla_h u_0\,\overline{u_0}
\big),
\]
we obtain
\begin{equation}\label{eq:pre-boundary-energy}
\Theta(kh)\big(
|u|_{1,h}^2
-(
\nabla_h u_{n-1}\,\overline{u_n}
-\nabla_h u_0\,\overline{u_0}
)
\big)
-k^2\|u\|_{0,h}^2
=
(f,u)_{h}.
\end{equation}
We next use the homogeneous discrete radiation conditions $D_k^+u_0 = D_k^-u_n = 0$. Then we deduce from \eqref{eqn:rewrite D_k+}-\eqref{eqn:rewrite D_k-} that $B(\ii kh)\,\nabla_h u_0=\ii k\,u_0$ and $B(\ii kh)\,\nabla_h u_{n-1}=-\ii k\,u_n$.
Substitution into \eqref{eq:pre-boundary-energy} yields
\begin{equation}\label{eq:complex-energy}
\Theta(kh)\bigg(
|u|_{1,h}^2
+\frac{\ii k}{B(\ii kh)}\big(|u_n|^2+|u_0|^2\big)
\bigg)
-k^2\|u\|_{0,h}^2
=
(f,u)_{h}.
\end{equation}
Finally, since $B(z)=z/(e^z-1)$, we have
\[
\Re\bigg( \Theta(s)\frac{\ii s}{B(\ii s)} \bigg) 
=
\Re\big( \Theta(s) ( e^{\ii s}-1) \big)
= \Theta(s) (\cos (s)-1)
=  -\frac{s^2}{2}.
\]
Using this formula and taking the real part of  \eqref{eq:complex-energy} gives \eqref{eq:real-part-energy}.
\end{proof}

We combine Corollary~\ref{rem:low-frequency-coercive} and Lemma~\ref{lem:complex-energy} to obtain global $k$-explicit
$L^2$ and discrete $H^1$ bounds for the homogeneous-boundary problem.

\begin{theorem}[Uniform stability for homogeneous radiation conditions]
\label{thm:homogeneous-uniform-stability}
The BPF scheme \eqref{eq:bpf-hom-bc} with homogeneous radiation conditions admits a unique solution.
For $s:=kh$ and $ t:=kL$, let $A_0(s,t)$ be
defined in \eqref{eq:def-const-A0}.
Then, for all
$kh\notin \pi\mathbb Z$,
\begin{equation}\label{eq:uniform-hom}
k\|u\|_{0,h}
\le
A_0(kh,kL)\, \|f\|_{0,h},
\qquad
\sqrt{\Theta}\,|u|_{1,h}
\le
A_0(kh,kL)\, \|f\|_{0,h}.
\end{equation}
\end{theorem}

\begin{proof}
We use Lemma~\ref{lem:complex-energy} and Corollary~\ref{rem:low-frequency-coercive}.
From \eqref{eq:real-part-energy} and \eqref{eq:low-frequency-coercive}, we eliminate the $H^1$-seminorm to derive
\begin{equation}
\label{eq:pre-L2-balance}
\big(1+\cos(kh)\big)\,\bigg(
k^2\|u\|_{0,h}^2
+
\frac{k^2h}{2}(|u_0|^2+|u_n|^2)
\bigg)
\le
\frac{L^2}{2\Theta}\,\|f\|_{0,h}^2 
- \cos(kh) \Re(f,u)_{h}.
\end{equation}
The Cauchy--Schwarz and Young inequalities with
\(
\varepsilon=2\cos^2(kh/2)\,k^2
\)
yields
\begin{equation}\label{eq:est-fu}
|\Re(f,u)_{h}|
\le
\frac{1}{4k^2}\sec^2\!\Big(\frac{kh}{2}\Big)\,\|f\|_{0,h}^2
+
\cos^2\!\Big(\frac{kh}{2}\Big)\,k^2\|u\|_{0,h}^2.
\end{equation}
Substituting \eqref{eq:est-fu} into \eqref{eq:pre-L2-balance} and dividing by $1+\cos(kh)=2\cos^2(kh/2)$, we obtain
\begin{equation*}
k^2\|u\|_{0,h}^2
\le 
\frac12
\bigg(
\frac{L^2}{2\Theta}\,\sec^2\!\Big(\frac{kh}{2}\Big)
+
\frac{1}{4k^2}\sec^4\!\Big(\frac{kh}{2}\Big)
\bigg)\|f\|_{0,h}^2
+
\frac12 k^2\|u\|_{0,h}^2,
\end{equation*}
which gives the discrete $L^2$ norm estimate in \eqref{eq:uniform-hom}.
To estimate the discrete $H^1$-seminorm, we return to \eqref{eq:real-part-energy} and \eqref{eq:low-frequency-coercive}, and eliminate the $L^2$ norm to find that
\begin{equation}
\label{eq:pre-H1-balance}
\Theta\big(1+\cos(kh)\big)\, |u|_{1,h}^2
\le
\frac{L^2}{2\Theta}\,\|f\|_{0,h}^2 
+ \Re(f,u)_{h}.
\end{equation}
Combining \eqref{eq:pre-H1-balance} with \eqref{eq:est-fu}, and dividing by $1+\cos(s)=2\cos^2(s/2)$,
\begin{equation*}
\Theta |u|_{1,h}^2
\le
\frac12
\bigg(
\frac{L^2}{2\Theta}\,\sec^2\!\Big(\frac{kh}{2}\Big)
+
\frac{1}{4k^2}\sec^4\!\Big(\frac{kh}{2}\Big)
\bigg)\|f\|_{0,h}^2
+
\frac12 k^2\|u\|_{0,h}^2.
\end{equation*}
Applying the discrete $L^2$ estimate in \eqref{eq:uniform-hom} gives the $H^1$-seminorm estimate.

It remains to show the well-posedness.
If $f\equiv0$, then
\eqref{eq:uniform-hom} implies $u\equiv0$, so the full discrete problem \eqref{eq:bpf-hom-bc} is unique.
Since the BPF discretization defines a square linear system in the
finite-dimensional space $\mathbb C^{n+1}$, uniqueness implies
existence.
Therefore the full discrete problem \eqref{eq:bpf-hom-bc} admits a unique solution, and the
proof of Theorem~\ref{thm:homogeneous-uniform-stability} is complete.
\end{proof}

\subsection{Lifting and proof of Theorem~\ref{thm:main-stability}}
\label{sec:energy-error-general}

We now return to the full BPF discretization
\eqref{discretization SG} with general boundary data
$g_0,g_L\in\mathbb C$.
The key observation is that, by
Proposition~\ref{lem:exact-bdry}, the homogeneous Helmholtz component
generated by the boundary data is reproduced exactly by the BPF scheme.
This allows us to reduce the general problem to the homogeneous
discrete radiation problem treated in
Theorem~\ref{thm:homogeneous-uniform-stability}.
We begin with a discrete estimate for the exact plane-wave lifting. 

\begin{lemma}[Plane-wave lifting]
\label{lem:plane-wave-lifting}
Let
\[
u^{\rm pw}(x)=\alpha e^{\ii kx}+\beta e^{-\ii kx},
\qquad \alpha=\frac{g_L}{2\ii k\,e^{\ii kL}},
\qquad
\beta=-\frac{g_0}{2\ii k},
\]
so that $u^{\rm pw}(x)$ with $x\in[0,L]$ satisfies the homogeneous Helmholtz equation with
impedance boundary data $g_0,g_L$.
Define the sampled grid function $v_i:=u^{\rm pw}(x_i)$ for $i=0,\dots,n$.
Then $v=\{v_i\}_{i=0}^n$ satisfies
\begin{equation}\label{eq:lifting-bound}
k\|v\|_{0,h}
\le \frac{\sqrt L}{2}\,(|g_0|+|g_L|),
\qquad
\sqrt{\Theta}\,|v|_{1,h}
\le \frac{\sqrt L}{2}\,(|g_0|+|g_L|).
\end{equation}
\end{lemma}

\begin{proof}
For the $L^2$ bound, we use $|v_i|\le |\alpha|+|\beta|$
and thus $\|v\|_{0,h}^2=h\sum_{i=1}^{n-1} |v_i|^2
\le L\,(|\alpha|+|\beta|)^2$.
The $L^2$ estimate in \eqref{eq:lifting-bound} then follows from
\begin{equation}\label{eqn:bound-alpbeta}
|\alpha|+|\beta|
\le \frac{|g_0|+|g_L|}{2k}.
\end{equation}

For the discrete $H^1$-seminorm, we compute for
$i=0,\dots,n-1$,
\[
(\nabla_h v)_i
=
\frac{\alpha(e^{\ii k(x_i+h)}-e^{\ii kx_i})
+\beta(e^{-\ii k(x_i+h)}-e^{-\ii kx_i})}{h}.
\]
Using $|e^{\ii\theta}-1|=2|\sin(\theta/2)|$, we find that
\[
|v|_{1,h}^2
=
h\sum_{i=0}^{n-1}|(\nabla_h v)_i|^2
\le
\frac{4\sin^2(kh/2)}{h^2}\,L\,(|\alpha|+|\beta|)^2.
\]
It follows from \eqref{eqn:bound-alpbeta} and the expression $\Theta(s) = s^2/(4\sin^2(s/2))$ that
\[
|v|_{1,h}^2
\le
\frac{L}{4\Theta(kh)}\,(|g_0|+|g_L|)^2,
\]
which is equivalent to the $H^1$-seminorm estimate in \eqref{eq:lifting-bound}.
\end{proof}

We now prove the main stability theorem stated at
the beginning of this section.

\begin{proof}[Proof of Theorem~\ref{thm:main-stability}]
Let $u_h$ denote the solution to the BPF scheme
\eqref{discretization SG} with source term $f$ and impedance data
$g_0,g_L$, and let $v$ be the plane-wave lifting constructed in
Lemma~\ref{lem:plane-wave-lifting}. We deduce from
Proposition~\ref{lem:exact-bdry} that
\[
D_k^-D_k^+v=0,
\qquad
(D_k^+v)_0=m(kh)g_0,
\qquad
(D_k^-v)_n=m(kh)g_L.
\]
Define the remainder $w:=u_h-v$. By linearity of the BPF scheme, $w$ satisfies the BPF scheme \eqref{eq:bpf-hom-bc} with homogeneous radiation conditions.
Applying the triangle inequality to $u_h=w+v$ and combining the estimate of $w$ from Theorem~\ref{thm:homogeneous-uniform-stability} and the estimate of $v$ from Lemma~\ref{lem:plane-wave-lifting}, we conclude \eqref{eq:main-L2-general} and
\eqref{eq:main-H1-general}.

It remains to show the well-posedness. The existence of solution $u_h$ is ensured by our construction.
If $f\equiv0$ and $g_0=g_L=0$, then
\eqref{eq:main-L2-general}--\eqref{eq:main-H1-general} imply $u_h\equiv0$, so the solution is unique.
Therefore the full discrete problem admits a unique solution, and the
proof of Theorem~\ref{thm:main-stability} is complete.
\end{proof}

\section{Consistency and convergence}\label{sec:consistency}

In this section, we establish the consistency of the BPF discretization
and derive the corresponding convergence estimates.
Let $u$ denote the exact solution of \eqref{eq:Helmholtz}--\eqref{eq:bc}
and let $u_h=\{u_i\}_{i=0}^n$ be the discrete BPF solution to
\eqref{discretization SG}.
Our goal is to obtain $k$-explicit second-order bounds for the grid error $e_i:=u_i-u(x_i)$ in the
discrete $L^2$ norm and $H^1$-seminorm.

\subsection{Residual equation and main results}

We begin by inserting the exact solution into the discrete BPF
operator.
This gives rise to the interior and boundary residuals that govern the error equation. For the interior nodes $i=1,\dots,n-1$, we define
the interior residual by
\begin{equation}\label{eq:def-tau-operator}
\tau_i (u)
:=
\Theta(kh)\,(\Delta_h u)(x_i)+k^2u(x_i)-f(x_i).
\end{equation}
Using the differential equation \eqref{eq:Helmholtz},
this can be rewritten equivalently as
\begin{equation}\label{eq:def-tau-explicit}
\tau_i(u)
=
\Theta(kh)\,(\Delta_h u)(x_i)-u''(x_i).
\end{equation}
At the boundaries, we define the residuals by
\begin{equation}\label{eq:def-beta-operator}
\beta_0(u) :=\frac{1}{m(kh)}(D_k^+u)_0-g_0,
\qquad
\beta_L(u) :=\frac{1}{m(kh)}(D_k^-u)_n- g_L.
\end{equation}
Using Proposition \ref{prop:factorization} in Appendix A, together with the impedance boundary
conditions in \eqref{eq:bc}, we may write them explicitly as
\begin{equation}\label{eq:def-beta0-explicit}
\begin{aligned}
\beta_0(u)
&=
\frac{k}{\sin(kh)}
\Big(u(h)-e^{\ii kh}u(0)\Big)
-\big(u'(0)-\ii ku(0)\big),
\\
\beta_L(u)
&=
\frac{k}{\sin(kh)}
\Big(e^{\ii kh}u(L)-u(L-h)\Big)
-\big(u'(L)+\ii ku(L)\big).
\end{aligned}
\end{equation}
For simplicity, when $u$ is the exact solution to \eqref{eq:Helmholtz}--\eqref{eq:bc}, we omit the arguments:
\begin{equation}\label{eq:w-beta-simple}
    \tau_i := \tau_i(u),\qquad 
    \beta_0 := \beta_0(u),\qquad 
    \beta_L := \beta_L(u).
\end{equation}
With these definitions, the grid error satisfies a discrete residual
equation of exactly the same form as the BPF scheme.

\begin{lemma}[Error equation]
Let $u$ be the exact solution to \eqref{eq:Helmholtz}--\eqref{eq:bc}, and let
$u_h$ be the discrete BPF solution to \eqref{discretization SG}.
Then the grid error $e_i:=u(x_i)-u_i$ for $i=0,\dots,n$
satisfies
\begin{equation*}
\left\{
\begin{aligned}
D_k^-D_k^+e_i = \tau_i&,\qquad i=1,\dots,n-1,\\
D_k^+e_0 = m(kh)\beta_0&,\qquad 
D_k^-e_n = m(kh)\beta_L,
\end{aligned}
\right.
\end{equation*}
where $\tau_i$, $\beta_0$, and $\beta_L$ are given by
\eqref{eq:def-tau-explicit}--\eqref{eq:w-beta-simple}.
\end{lemma}

We now state the main results of this section.
The first theorem provides $k$-explicit bounds for the interior and
boundary residuals.

\begin{theorem}[Consistency and residual bounds]
\label{thm:residual-bounds}
Assume that $f\in H^3(0,L)\cap H^2_0(0,L)$, and let
$\tau=\{\tau_i\}_{i=1}^{n-1}$, $\beta_0$, and $\beta_L$ be
defined by \eqref{eq:def-tau-operator} and
\eqref{eq:def-beta-operator}.
Then, for any $kh\notin \pi\mathbb Z$, the residuals satisfy the second-order estimates
\begin{align}
\|\tau\|_{0,h}
&\le
L\,\Theta(kh)\,\frac{h^2}{12}\,\|f^{(3)}\|_{L^2(0,L)},
\label{eq:tau-bound-main}
\\
|\beta_0|+|\beta_L|
&\le
2\sqrt{L\,\Theta(kh)}\,\bigg| \sec\!\bigg(\frac{kh}{2}\bigg)\bigg| \,
\frac{h^2}{6}\, 
\|f''\|_{L^2(0,L)}.
\label{eq:beta-bound-main}
\end{align}
In particular, for any fixed $s_0<\pi$, the constants in
\eqref{eq:tau-bound-main}--\eqref{eq:beta-bound-main} are uniformly
bounded with respect to $k$ whenever $kh\le s_0$.
\end{theorem}

Combining the residual equation with the stability result of
Theorem~\ref{thm:main-stability}, we obtain the following convergence
theorem.

\begin{theorem}[Convergence and error estimate]
\label{thm:main-convergence}
Assume that $f\in H^3(0,L)\cap H^2_0(0,L)$.
For $s:=kh$ and $ t:=kL$, let $A_0(s,t)$ be
defined by \eqref{eq:def-const-A0}.
Then, for any $kh\notin \pi\mathbb Z$, the grid error $e_h=\{e_i\}_{i=0}^n$ satisfies
\begin{align*}
k\|e_h\|_{0,h}
&\le
\Theta\,A_0(kh,kL)
\bigg(
L\,\frac{h^2}{12}\,\|f^{(3)}\|_{L^2(0,L)}
+
\frac{h^2}{3}\,\|f''\|_{L^2(0,L)}
\bigg),
\\
\sqrt{\Theta}\,|e_h|_{1,h}
&\le
\Theta\,A_0(kh,kL)
\bigg(
L\,\frac{h^2}{12}\,\|f^{(3)}\|_{L^2(0,L)}
+
\frac{h^2}{3}\,\|f''\|_{L^2(0,L)}
\bigg).
\end{align*}
In particular, for $s_0<\pi$, the method is second-order
convergent in $\|\cdot\|_{0,h}$ and $|\cdot|_{1,h}$, with
constants uniformly bounded with respect to $k$ whenever
$kh\le s_0$ and $kL\ge \pi$.
\end{theorem}

The remainder of this section is devoted to proving the
$k$-explicit consistency estimates stated in
Theorem~\ref{thm:residual-bounds}. As a Taylor expansion of the leading interior and boundary consistency terms does not provide a $k$-uniform control of the remainders, we choose a Fourier-based approach inspired by \cite{gander2025fourier}. In the present setting, the impedance boundary closures introduce additional difficulties, which are overcome by lifting the Helmholtz kernel component.


\subsection{Exact kernel lifting and sine expansion}
\label{subsec:kernel-lifting}

We begin the proof of Theorem~\ref{thm:residual-bounds}.
The key observation is that the BPF discretization is exact on the
Helmholtz kernel $\mathrm{span}\{e^{\pm \ii kx}\}$. This allows us to separate from the exact solution a homogeneous
oscillatory component that contributes no residual, and to reduce the
consistency analysis to a zero-trace remainder.

\begin{lemma}[Exact kernel lifting]
\label{lem:kernel-lifting}
Assume that $\sin(kL)\neq0$.
Then there exists a unique function $b\in C^\infty([0,L])$ such that
\begin{equation}
\label{eq:kernel-lifting}
-b''-k^2 b = 0 \quad \text{in }(0,L),
\qquad
b(0)=u(0),\qquad b(L)=u(L).
\end{equation}
Moreover, $b$ belongs to $\mathrm{span}\{\cos(kx),\sin(kx)\}
= \mathrm{span}\{e^{\pm \ii kx}\}$.
\end{lemma}

\begin{proof}
The general solution to the homogeneous Helmholtz equation $-b''-k^2b=0$
is of the form $b(x)=A\cos(kx)+B\sin(kx)$ with $A$, $B\in\mathbb C$.
The boundary conditions in \eqref{eq:kernel-lifting} give $A=u(0)$ and $A\cos(kL)+B\sin(kL)=u(L)$.
Since $\sin(kL)\neq0$, this system is nonsingular and
therefore determines a unique pair $(A,B)$.
The representation above also shows that
$b\in \mathrm{span}\{\cos(kx),\sin(kx)\}$.
\end{proof}

\begin{remark}[Resonant wavenumbers]
The assumption $\sin(kL)\neq0$ is imposed only to simplify the derivation.
If $\sin(kL)=0$ (i.e.\ $kL=n\pi$ for some integer $n$), one may apply a limiting procedure.
The residual estimates obtained below extend to the resonant case
$\sin(kL)=0$ by continuity with respect to $k$.
\end{remark}

We now define the lifted remainder $w:=u-b$.
Since $u$ solves \eqref{eq:Helmholtz} and $b$ satisfies
\eqref{eq:kernel-lifting}, it follows that $w(0) = u(0) - b(0) = 0$ and  $w(L) = u(L) - b(L) = 0$. Then, the lifted remainder $w$ satisfies the Dirichlet problem
\begin{equation}\label{eq:w-hom-dirichlet}
-w''-k^2w=f \quad \text{in }(0,L),
\qquad
w(0)=w(L)=0.
\end{equation}
The relevance of this decomposition is that the kernel part $b$
produces no truncation error.
Indeed, by Proposition~\ref{lem:exact-bdry}, the BPF scheme is exact on
the plane-waves $e^{\pm \ii kx}$ and hence, by linearity, on the entire
space $\mathrm{span}\{e^{\pm \ii kx}\}$.
Therefore, $\tau_i(b)=0$ for $i=1,\dots,n-1$ and $\beta_0(b)=0$, $\beta_L(b)=0$.
The residuals of the exact solution reduce to those of
the lifted remainder:
\begin{equation*}
\tau_i(u)=\tau_i(w),\qquad i=1,\dots,n-1,
\qquad
\beta_0(u)=\beta_0(w),\qquad
\beta_L(u)=\beta_L(w).
\end{equation*}
Since $w(0)=w(L)=0$, the boundary residuals \eqref{eq:def-beta0-explicit}-\eqref{eq:w-beta-simple} admit the simplified forms
\begin{align}
\beta_0 = \beta_0(w)
&=
\frac{k}{\sin(kh)}\,w(h)-w'(0),
\label{eq:beta0-lifted}
\\
\beta_L = \beta_L(w)
&=
\frac{-k}{\sin(kh)}\,w(L-h)-w'(L).
\label{eq:betaL-lifted}
\end{align}
Thus, the truncation errors are determined by the zero-trace
solution $w$ of \eqref{eq:w-hom-dirichlet}.

To expand the lifted zero-trace component \(w\) in the since basis, we introduce
\begin{equation}\label{eqn:L2-basis}
\phi_n(x):=\sqrt{\frac{2}{L}}\sin(\xi_n x),
\qquad
\xi_n:=\frac{n\pi}{L},
\qquad n\in\mathbb N.
\end{equation}
Then $\{\phi_n\}_{n\ge1}$ is an orthonormal basis of $L^2(0,L)$, and
\begin{equation}\label{eqn:expand-w}
w(x)=\sum_{n\ge1}\widehat w_n\,\phi_n(x),
\qquad
\widehat w_n:=(w,\phi_n)_{L^2(0,L)}.
\end{equation}
Similarly, since $f\in L^2(0,L)$, we may write
\[
f(x)=\sum_{n\ge1}\widehat f_n\,\phi_n(x),
\qquad
\widehat f_n:=(f,\phi_n)_{L^2(0,L)}.
\]


In the next two subsections, we use the Fourier-based analysis to derive separate $k$-explicit
estimates for the interior residual $\tau$ and the boundary residuals
$\beta_0$ and $\beta_L$.

\subsection{Interior residual estimate}

We estimate the interior residual $\tau_i(w)$ by adopting a ``semi-discrete'' viewpoint. For this, we introduce the continuous and semi-discrete Helmholtz operators on $C^2[0,L]$:
\begin{equation}
\label{eq:def-cont-disc-operators}
\mathcal L := -\partial_{xx}-k^2,
\qquad
\mathcal L_h := -\Theta(kh)\,\Delta_h-k^2.
\end{equation}
Then the interior truncation error \eqref{eq:def-tau-explicit} can be extended to a pointwise residual,
\begin{equation}
\label{eq:def-tau-h}
\tau_h(x):=(\mathcal L_h-\mathcal L)w(x),
\qquad x\in(0,L),
\end{equation}
so that, by construction,
\begin{equation}
\label{eq:tau-grid-from-tau-h}
\tau_i(w) = \tau_i(u) =\tau_h(x_i),\qquad i =1,\dots,n-1.
\end{equation}

By \eqref{eq:tau-grid-from-tau-h}, it suffices to analyze the pointwise
residual $\tau_h(x)$. We use Fourier multiplier analysis. Note that the sine basis diagonalizes the continuous and semi-discrete Helmholtz operators.
Indeed, from \eqref{eqn:L2-basis},
\begin{equation}
\label{eq:cont-symbol}
\mathcal L\phi_n
=
(-\phi_n''-k^2\phi_n)
=
(\xi_n^2-k^2)\phi_n
=:\lambda(\xi_n)\phi_n.
\end{equation}
Moreover, since
\[
(\Delta_h\phi_n)(x)
=
-\frac{4}{h^2}\sin^2\!\Big(\frac{\xi_n h}{2}\Big)\phi_n(x),
\]
it follows from \eqref{eq:def-cont-disc-operators} that
\begin{equation}
\label{eq:def-lambda-discrete}
\mathcal L_h\phi_n
=
\bigg(
\Theta(kh)\,\frac{4}{h^2}\sin^2\!\Big(\frac{\xi_n h}{2}\Big)-k^2
\bigg)\phi_n
=:
\widetilde\lambda_h(\xi_n)\phi_n.
\end{equation}
Finally, testing the Helmholtz equation \eqref{eq:w-hom-dirichlet} with $\phi_n$ gives the modal relation
\begin{equation}
\label{eq:modal-relation}
\lambda(\xi_n)\,\widehat w_n=\widehat f_n,
\qquad\text{that is,}\qquad
\widehat w_n=\frac{\widehat f_n}{\lambda(\xi_n)},
\qquad n\ge1.
\end{equation}

We begin with the modal representation of $\tau_h(x)$ in the sine basis.

\begin{lemma}[Modal representation of the interior residual]
\label{lem:modal-interior-residual}
Let \(u\) be the exact solution to \eqref{eq:Helmholtz}--\eqref{eq:bc}, let
\(b\) and \(w=u-b\) be defined as in
Section~\ref{subsec:kernel-lifting}, and let \(\tau_h\) be given by
\eqref{eq:def-tau-h}.
Then 
\[
\tau_h(x)=\sum_{n\ge1}\widehat{\tau}_{h,n}\,\phi_n(x),
\]
where for $n\ge 1$,
\begin{align}
\label{eq:modal-tau}
\widehat{\tau}_{h,n}
&=
M_h(\xi_n;h,k)\,\widehat f_n, \\
\label{eq:def-mh-theta}
M_h(\xi;h,k)
&:=
\frac{\widetilde\lambda_h(\xi)-\lambda(\xi)}{\lambda(\xi)}
=
\frac{
\Theta(kh)\,\dfrac{4}{h^2}\sin^2\!\bigl(\tfrac{\xi h}{2}\bigr)-\xi^2
}{
\xi^2-k^2
}.
\end{align}
Here, $\lambda(\xi)$ and $\widetilde\lambda_h(\xi)$ are the Fourier symbols of \(\mathcal L\) and  \(\mathcal L_h\) defined in \eqref{eq:cont-symbol} and \eqref{eq:def-lambda-discrete}, respectively.
\end{lemma}

\begin{proof}
It follows from definition \eqref{eq:def-tau-h} of $\tau_h$ and expansion \eqref{eqn:expand-w} of $w$ that
\[
\tau_h=(\mathcal L_h-\mathcal L)w
= \sum_{n\ge1}\widehat w_n (\mathcal L_h-\mathcal L)\phi_n
=  \sum_{n\ge1}\widehat w_n \big(\widetilde\lambda_h(\xi_n) - \lambda(\xi_n)\big) \phi_n,
\]
where we used \eqref{eq:cont-symbol} and \eqref{eq:def-lambda-discrete}.
We apply \eqref{eq:modal-relation} to obtain
\[
\widehat{\tau}_{h,n}
=
\bigl(\widetilde\lambda_h(\xi_n)-\lambda(\xi_n)\bigr)\widehat w_n
=
\frac{\widetilde\lambda_h(\xi_n)-\lambda(\xi_n)}
{\lambda(\xi_n)}\,\widehat f_n,
\]
which is exactly \eqref{eq:modal-tau}--\eqref{eq:def-mh-theta}.
\end{proof}

The next lemma gives a \(k\)-explicit bound for the multiplier
\(M_h\).

\begin{lemma}[Uniform bound for the interior multiplier]
\label{lem:interior-multiplier-bound}
It holds for all $kh\notin \pi\mathbb Z$ and \(\xi>0\) that
\begin{equation}
\label{eq:mh-bound}
|M_h(\xi;h,k)|
\le
\Theta(kh)\,\frac{h^2}{12}\,\xi^2.
\end{equation}
\end{lemma}

\begin{proof}
Set $A:=\xi h/2$ and $B:=kh/2$.
The multiplier \eqref{eq:def-mh-theta} can be rewritten as
\[
M_h(\xi;h,k)
=
\frac{\Theta\sin^2 A-A^2}{A^2-B^2}.
\]
Introducing the function $g(t):=\sin^2(\sqrt t)/t$ for $t>0$,
we have $\sin^2A=A^2 g(A^2)$ and $\sin^2B=B^2 g(B^2)$. Then, in view of $\Theta = 1/g(B^2)$ and  the mean value theorem,
\[
|M_h(\xi;h,k)|
=
\bigg|A^2\Theta(kh)\,
\frac{g(A^2)-g(B^2)}{A^2-B^2}\bigg|
\le A^2\Theta(kh)\,\sup_{t>0} |g'(t)|.
\]
We refer to the Appendix for a proof that $\sup_{t>0} |g'(t)| \le 1/3$.
Hence
\[
|M_h(\xi;h,k)|
\le
A^2\Theta(kh)\cdot \frac13
=
\Theta(kh)\,\frac{h^2}{12}\,\xi^2,
\]
which proves \eqref{eq:mh-bound}. 
\end{proof}

We can now derive the \(k\)-explicit bound for the interior residual.

\begin{theorem}[Interior residual estimate]
Assume that \(f\in H^3(0,L)\cap H_0^2(0,L)\).
Then the interior residual
\(\tau=\{\tau_i\}_{i=1}^{n-1}\) satisfies
\begin{equation}
\label{eq:tau-bound-main-subsection}
\|\tau\|_{0,h}
\le
L\,\Theta(kh)\,\frac{h^2}{12}\,\|f^{(3)}\|_{L^2(0,L)},
\end{equation}
so the interior part of Theorem~\ref{thm:residual-bounds} follows.
Moreover, under the fixed-resolution condition \(kh\le s_0<\pi\), the
constant in \eqref{eq:tau-bound-main-subsection} is uniformly bounded
with respect to \(k\).
\end{theorem}

\begin{proof}
We know from Lemma~\ref{lem:modal-interior-residual} that
\[
\tau_h(x)
=
\sum_{n\ge1}M_h(\xi_n;h,k)\,\widehat f_n\,\phi_n(x).
\]
Using \(|\phi_n(x)|\le \sqrt{2/L}\) and the Cauchy--Schwarz inequality, we obtain
\begin{align*}
|\tau_h(x)|
&\le
\sqrt{\frac{2}{L}}
\sum_{n\ge1}|M_h(\xi_n;h,k)|\,|\widehat f_n|
\\
&\le
\sqrt{\frac{2}{L}}
\bigg(
\sum_{n\ge1}|M_h(\xi_n;h,k)|^2\,\xi_n^{-6}
\bigg)^{1/2}
\bigg(
\sum_{n\ge1}\xi_n^6|\widehat f_n|^2
\bigg)^{1/2}.
\end{align*}
By Lemma~\ref{lem:interior-multiplier-bound},
\[
\sum_{n\ge1}|M_h(\xi_n;h,k)|^2\,\xi_n^{-6}
\le
\Theta(kh)^2\frac{h^4}{144}\sum_{n\ge1}\xi_n^{-2}.
\]
Since \(\xi_n=n\pi/L\),
\begin{equation}\label{eq:sum-n^2}
\sum_{n\ge1}\xi_n^{-2}
=
\frac{L^2}{\pi^2}\sum_{n\ge1}\frac1{n^2}
=
\frac{L^2}{6}.
\end{equation}
Because of \(f\in H^3(0,L)\cap H_0^2(0,L)\), Parseval's identity for the sine basis gives
$
\sum_{n\ge1}\xi_n^6|\widehat f_n|^2
=
\|f^{(3)}\|_{L^2(0,L)}^2.
$
Combining these estimates, we find that
\[
|\tau_h(x)|
\le
\sqrt{\frac{2}{L}}
\frac{\Theta(kh)\,h^2}{12}
\bigg(\frac{L^2}{6}\bigg)^{1/2}
\,
\|f^{(3)}\|_{L^2(0,L)}
=
\sqrt{\frac{L}{3}}\,
\frac{\Theta(kh)\,h^2}{12}\,
\|f^{(3)}\|_{L^2(0,L)}.
\]
Finally, using \eqref{eq:tau-grid-from-tau-h},
\begin{align*}
  \|\tau\|_{0,h} = \bigg(h\sum_{i=1}^{n-1}|\tau_h(x_i)|^2\bigg)^{1/2}
  \le \sqrt{L}\,\|\tau_h\|_{L^\infty(0,L)}
  \le \frac{L\,\Theta(kh)}{12\sqrt3}\,
  h^2\,\|f^{(3)}\|_{L^2(0,L)}.
\end{align*}
This proves \eqref{eq:tau-bound-main-subsection}.
\end{proof}

\subsection{Boundary residual estimate}

We turn to the boundary residuals $\beta_0$, $\beta_L$, which are defined in \eqref{eq:def-beta0-explicit}--\eqref{eq:w-beta-simple}.
We first derive their modal representation.

\begin{lemma}[Modal representation of the boundary residuals]
\label{lem:modal-boundary-residual}
Let \(u\) be the exact solution to \eqref{eq:Helmholtz}--\eqref{eq:bc}, let \(w=u-b\), where $b$ is defined in Lemma \ref{lem:kernel-lifting}, and let \(\beta_0,\beta_L\) be
defined by \eqref{eq:beta0-lifted}--\eqref{eq:betaL-lifted}.
Then
\begin{align}
  & \beta_0=\sum_{n\ge1} B_h(\xi_n;h,k)\,\widehat f_n, \quad
  \beta_L=\sum_{n\ge1} (-1)^nB_h(\xi_n;h,k)\,\widehat f_n,
  \quad\mbox{where} \nonumber \\
  & \label{eq:def-bh}
  B_h(\xi;h,k) = \sqrt{\frac{2}{L}}\;\frac{h}{\sin (kh)}\;
  \frac{kh\sin (\xi h)-\xi h\sin (kh)}{(\xi h)^2-(kh)^2}.
\end{align}
\end{lemma}

\begin{proof}
We infer from expansion \eqref{eqn:expand-w} of $w$ in the sine basis that
\[
w'(0)=\sum_{n\ge1}\sqrt{\frac{2}{L}}\,\xi_n\,\widehat w_n,
\qquad
w(h)=\sum_{n\ge1}\sqrt{\frac{2}{L}}\sin(\xi_n h)\,\widehat w_n.
\]
Since \eqref{eq:modal-relation} gives $(\xi_n^2-k^2)\widehat w_n=\widehat f_n$,
we deduce from \eqref{eq:beta0-lifted} that
\begin{align*}
\beta_0
&=
\sum_{n\ge1}\sqrt{\frac{2}{L}}
\bigg(
\frac{k}{\sin(kh)}\sin(\xi_n h)-\xi_n
\bigg)\widehat w_n
\\
&=
\sum_{n\ge1}\sqrt{\frac{2}{L}}
\frac{
\frac{k}{\sin(kh)}\sin(\xi_n h)-\xi_n
}{
\xi_n^2-k^2
}\,\widehat f_n
=
\sum_{n\ge1} B_h(\xi_n;h,k)\,\widehat f_n.
\end{align*}
For the right endpoint, we use $\sin(\xi_n L)=0$ and 
$\cos(\xi_n L)=(-1)^n$ to write
\begin{align*}
w'(L) &=
\sum_{n\ge1}\sqrt{\frac{2}{L}}\,(-1)^n\xi_n\,\widehat w_n, \\
w(L-h) &=
\sum_{n\ge1}\sqrt{\frac{2}{L}}\sin\bigl(\xi_n(L-h)\bigr)\widehat w_n
=
-\sum_{n\ge1}\sqrt{\frac{2}{L}}\,(-1)^n\sin(\xi_n h)\,\widehat w_n,
\end{align*}
where we have used $\sin(a-b) = \sin(a)\cos(b)-\cos(a)\sin(b)$ in the last equality.
Substituting into \eqref{eq:betaL-lifted}, we obtain
\begin{equation*}
\beta_L
=
\sum_{n\ge1}\sqrt{\frac{2}{L}}\,(-1)^n
\Bigl(
\frac{k}{\sin(kh)}\sin(\xi_n h)-\xi_n
\Bigr)\widehat w_n
=
\sum_{n\ge1}(-1)^n\,B_h(\xi_n;h,k)\,\widehat f_n,
\end{equation*}
which finishes the proof.
\end{proof}

The next lemma gives a \(k\)-explicit bound for the boundary
multiplier.

\begin{lemma}[Uniform bound for the boundary multiplier]
\label{lem:boundary-multiplier-bound}
It holds for all $kh\notin \pi\mathbb Z$ and \(\xi>0\) that
\begin{equation*}
|B_h(\xi;h,k)|
\le
\sqrt{\frac{2\Theta(kh)}{L}}\;
\frac{\xi h^2}{6}\, \Big| \sec\!\Big(\frac{kh}{2}\Big)\Big|.
\end{equation*}
\end{lemma}

\begin{proof}
Starting from \eqref{eq:def-bh}, we rewrite \(B_h\) as
\[
B_h(\xi;h,k)
=
\sqrt{\frac{2}{L}}\;
\frac{hAB}{\sin B}\;
\frac{h(A^2)-h(B^2)}{A^2-B^2},
\qquad
A:=\xi h,\quad B:=kh,
\]
where $h(t):=\sin(\sqrt{t})/\sqrt{t}$ for $t>0$. We apply the mean-value theorem and the bound $\sup_{t>0} |h'(t)|\le 1/6$ from Lemma \ref{lem:appendix-g} in the Appendix:
\begin{align*}
  |B_h(\xi;h,k)| &\le\sqrt{\frac{2}{L}}\;
  \frac{hAB}{|\sin B|}\;\sup_{t>0} |h'(t)|
  \le \sqrt{\frac{2}{L}}\;\frac{\xi h^2}{6}\;\frac{kh}{|\sin (kh)|} \\
  &= \sqrt{\frac{2}{L}}\;\frac{\xi h^2}{6}\;
  \sqrt{\Theta(kh)}\, \Big| \sec\!\Big(\frac{kh}{2}\Big)\Big|.
\end{align*}
This ends the proof.
\end{proof}

We can now estimate the boundary residuals.

\begin{theorem}[Boundary residual estimate]
Assume that \(f\in H_0^2(0,L)\).
Then 
\begin{equation}
\label{eq:beta-bound-main-subsection}
|\beta_0|+|\beta_L|
\le
2\sqrt{L\,\Theta(kh)}\;
\Big| \sec\!\Big(\frac{kh}{2}\Big)\Big|\;
\frac{h^2}{6}\;
\|f''\|_{L^2(0,L)},
\end{equation}
so the boundary part of Theorem~\ref{thm:residual-bounds} follows.
Moreover, under the fixed-resolution condition \(kh\le s_0<\pi\), the
constant in \eqref{eq:beta-bound-main-subsection} is uniformly bounded
with respect to \(k\).
\end{theorem}

\begin{proof}
We conclude from Lemma~\ref{lem:modal-boundary-residual} and the Cauchy--Schwarz inequality that
\[
|\beta_0|
= \bigg|\sum_{n\ge1}B_h(\xi_n;h,k)\,\widehat f_n\bigg|
\le
\bigg(\sum_{n\ge1}|B_h(\xi_n;h,k)|^2\,\xi_n^{-4}\bigg)^{1/2}
\bigg(\sum_{n\ge1}\xi_n^4|\widehat f_n|^2\bigg)^{1/2}.
\]
It follows from Lemma~\ref{lem:boundary-multiplier-bound} and \eqref{eq:sum-n^2} that
\[
\sum_{n\ge1}|B_h(\xi_n;h,k)|^2\,\xi_n^{-4}
\le
\frac{2\Theta}{L}\,
\sec^2\!\Big(\frac{kh}{2}\Big)\,
\frac{h^4}{36}
\sum_{n\ge1}\xi_n^{-2}
=
\frac{L\Theta}{3}\,
\sec^2\!\Big(\frac{kh}{2}\Big)\,
\frac{h^4}{36}.
\]
Moreover, because of \(f\in H_0^2(0,L)\), Parseval's identity for the sine basis yields
$
\sum_{n\ge1}\xi_n^4|\widehat f_n|^2
=
\|f''\|_{L^2(0,L)}^2.
$
Combining these estimates, we obtain
\begin{equation}\label{eq:beta0-bound-sharp}
|\beta_0|
\le
\bigg(
\frac{L\Theta}{3}\,
\sec^2\!\Big(\frac{kh}{2}\Big)\,
\frac{h^4}{36}
\bigg)^{1/2}
\|f''\|_{L^2(0,L)}.
\end{equation}

For the right endpoint, Lemma~\ref{lem:modal-boundary-residual} shows that \(|\widetilde B_h(\xi_n;h,k)|=|B_h(\xi_n;h,k)|\) and thus, the same
argument yields the same bound as in \eqref{eq:beta0-bound-sharp}.
Finally, summing the estimates of $\beta_0$ and $\beta_L$ leads to
\eqref{eq:beta-bound-main-subsection}.
\end{proof}

\section{Numerical Experiments}\label{sec:numerical-experiments}

In this section, we present numerical experiments for the proposed BPF
scheme. The first two tests are designed to validate the theoretical
analysis. We then test the robustness of the BPF scheme in terms of nonsmooth sources. Finally, we compare various numerical methods.

We consider the relative $L^{\infty}$- and $V$-errors, defined by, respectively,
\begin{equation*}
    \|e\|_{L^{\infty}} = \frac{\|u - u_h\|_{L^{\infty}}}{\|u\|_{L^{\infty}}}, \quad
    \|e\|_{V} = \frac{\|u - u_h\|_{V}}{\|u\|_{V}},
\end{equation*}
where $\|u\|_V^2 = k^2 \|u\|_{0, h}^2 + |u|_{1,h}^2$, and $u$ is the reference solution.

\subsection{Exact resolution of a 1D plane-wave}
As shown in Section~\ref{subsec:symbol}, the BPF scheme exactly
reproduces plane-wave solutions. To illustrate this property, we
consider the exact $u(x)=2e^{\ii kx}+e^{-\ii kx}$ over the interval $(0,1)$.
The corresponding impedance are $g_0=u'(0)-\ii ku(0)=-2\ii k$ and $g_1=u'(1)+\ii ku(1)=4\ii ke^{\ii k}$.
We choose the wavenumber $k=2^7$ and the coarse mesh size $h=2^{-3}$.
With this setup, the computed solution agrees with the sampled exact
solution up to machine precision; more precisely, the absolute
$L^\infty$-error is $2.91\times 10^{-15}$. In
Figure~\ref{fig:1D-exact-solution}, we plot the real and imaginary
parts of the exact and numerical solutions. The figures show that the
BPF scheme reproduces the plane-wave accurately even on a very coarse
mesh.
\begin{figure}[ht]
    \centering
    \includegraphics[width=0.49\linewidth]{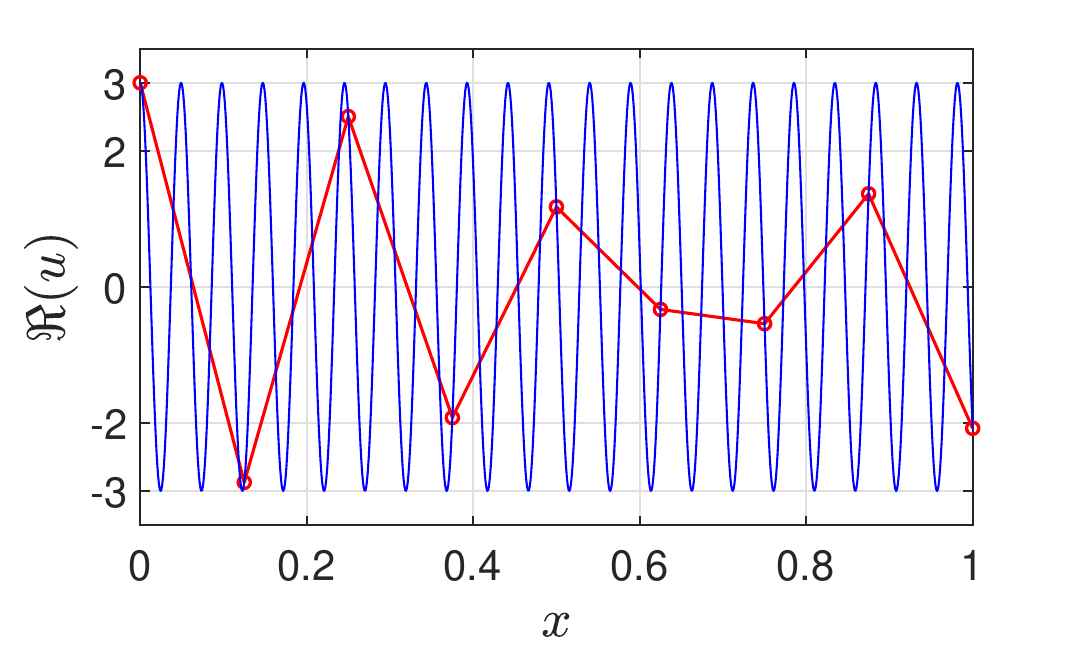}\includegraphics[width=0.49\linewidth]{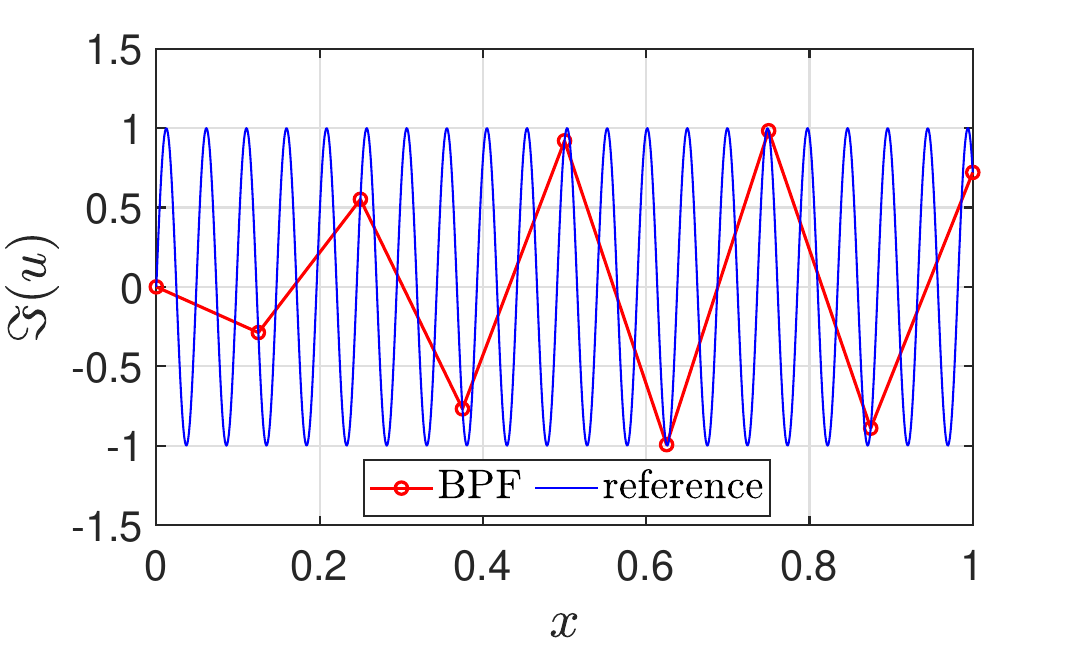}
    \caption{Comparison between the reference solution and the numerical
    solution for the plane-wave test on a coarse mesh.}
    \label{fig:1D-exact-solution}
\end{figure}

\subsection{A smooth manufactured-solution test}
To directly validate the convergence estimate of Theorem~4.3, we
consider the smooth manufactured solution
$u(x)=e^{\ii kx}+r(x)$, $r(x):=x^4(1-x)^4$ for $x\in(0,1)$.
The source term is defined by $f(x) = r''(x)+k^2r(x)$ or, more explicitly,
\begin{align*}
f(x) =  12x^2(1-x)^4 - 32x^3(1-x)^3 + 12x^4(1-x)^2 + k^2x^4(1-x)^4.
\end{align*}
The impedance boundary data are taken from the exact solution:
\[
g_0=u'(0)-\ii ku(0)=0,
\qquad
g_1=u'(1)+\ii ku(1)=2\ii k e^{\ii k}.
\]
Taking into account that $r$ has fourth-order zeros at both endpoints, it follows that $f(0)=f(1)=f'(0)=f'(1)=0$,
so this example is consistent with our regularity assumptions. Moreover, it holds that $\|f^{(\alpha)}\|_{2}\lesssim k^2$ for $\alpha= 2,3$.

We report the relative $L^\infty$- and $V$-errors
as the mesh is refined with $h=3^{-9},\,\dots,$ $3^{-5}$.
Because the $V$-norm of the exact solution does not depend on $h$ when $k$ is fixed, the relative $V$-norm error is the same as the absolute energy error when assessing the convergence rate in relation to $h$. Figure~\ref{fig:1D-convergence-smooth-f} shows that the BPF solution exhibits
second-order convergence under mesh refinement, in agreement with
Theorem~4.3.

\begin{figure}[ht]
    \centering
    \includegraphics[width=0.45\linewidth]{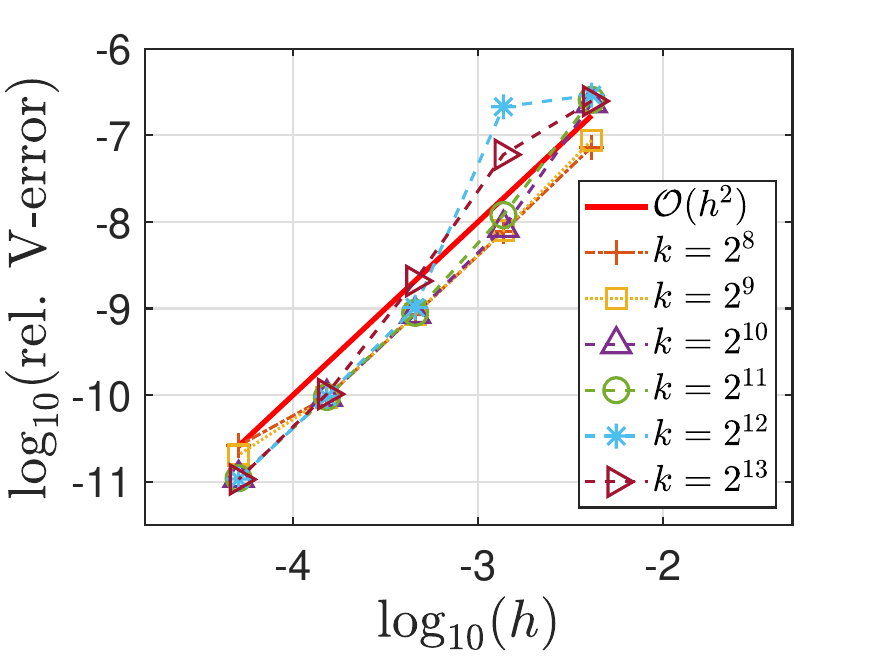}\includegraphics[width=0.45\linewidth]{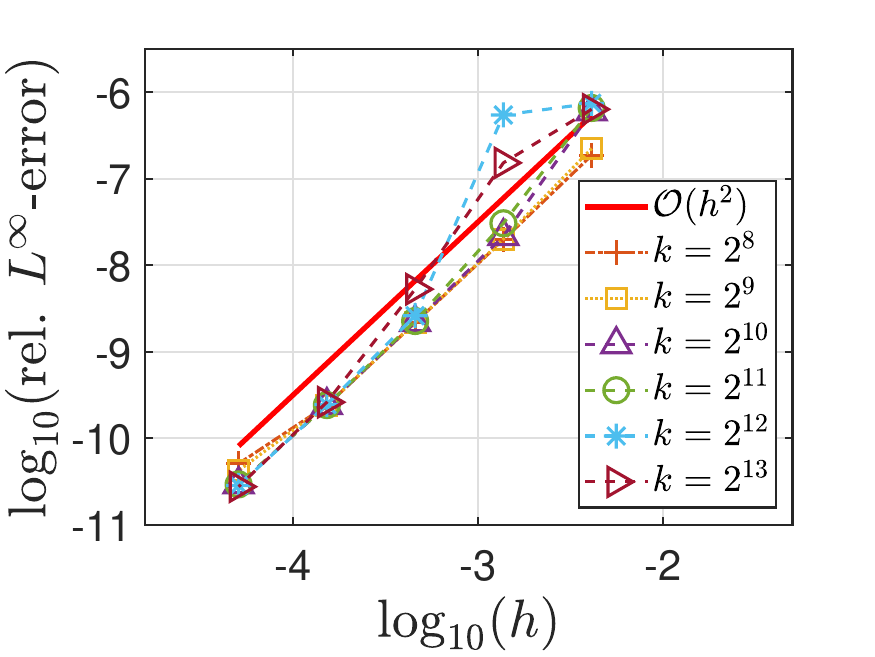}
    \caption{Results for the smooth-source test: convergence with
    respect to $h$ in the relative $V$-norm (left) and the relative
    $L^\infty$-norm (right).}
    \label{fig:1D-convergence-smooth-f}
\end{figure}

\subsection{A nonsmooth source test}

We consider a nonsmooth source term to illustrate the
robustness of the BPF scheme beyond the regularity assumptions required in Theorem~4.3. We take $f(x)=50$ for $|x-0.5|\le 1/9$ and $f(x)=0$ else,
together with the nonhomogeneous impedance boundary condition
\begin{equation}\label{equ:nontrival-Robin-boundary-1D}
    u'(0) - \ii k u(0) = 2, \qquad u'(1) + \ii k u(1) = \ii.
\end{equation}
Since an exact solution is not available, we compute a reference solution on the fine mesh $h=3^{-12}$. This example does not satisfy the smoothness assumptions of the theoretical analysis. We choose $h=3^{-10}\,\dots,\,3^{-5}$.
As shown in Figure~\ref{fig:1D-convergence-discontinuous-f}, the BPF
scheme still exhibits an approximately second-order convergence trend
with respect to $h$ in the relative $V$- and
$L^\infty$-norms, despite the low regularity of the source term.
\begin{figure}[ht]
    \centering
    \includegraphics[width=0.45\linewidth]{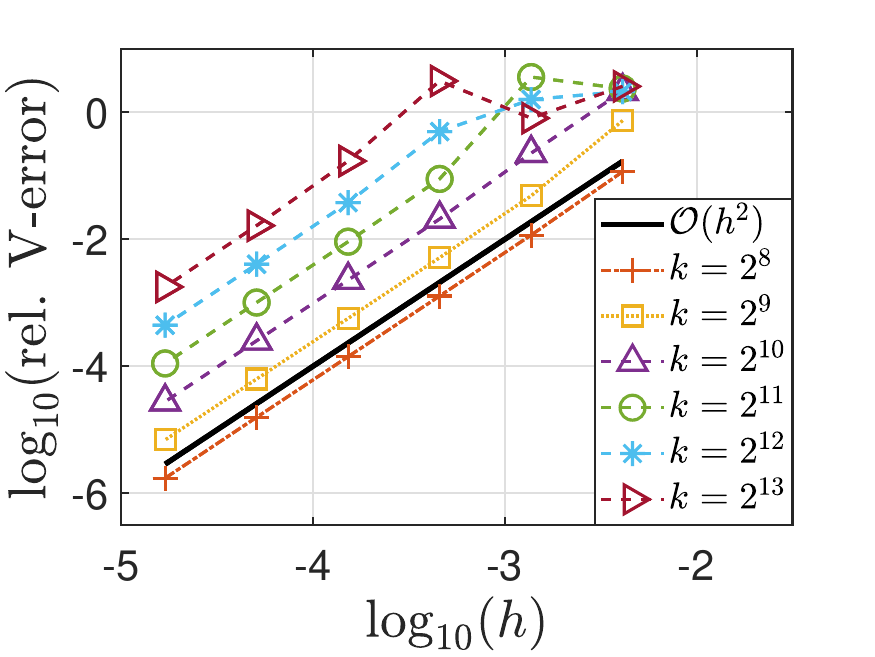}\includegraphics[width=0.45\linewidth]{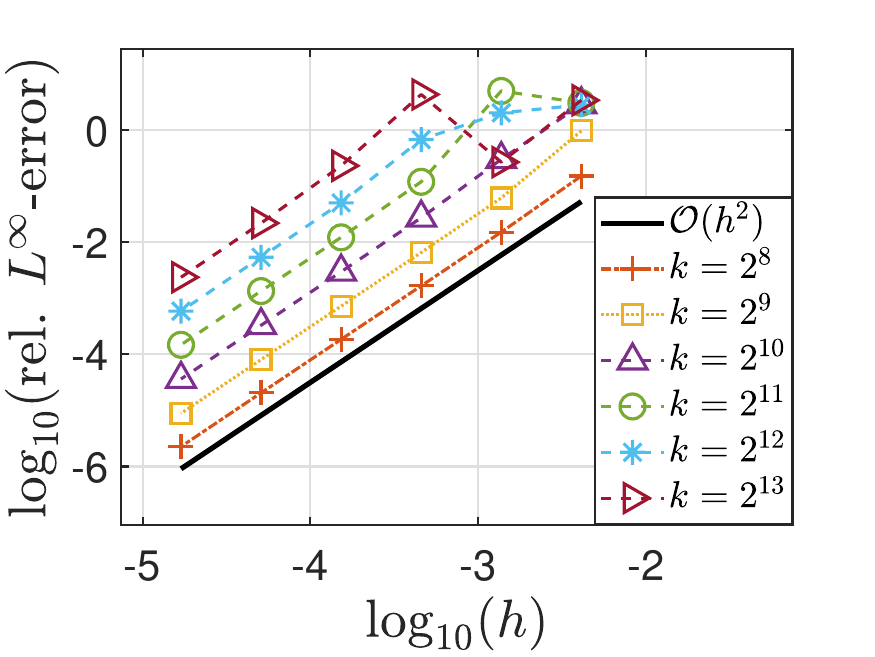}
    \caption{Results for the nonsmooth-source test: convergence with
    respect to $h$ in the relative $V$-norm (left) and the relative
    $L^\infty$-norm (right)}
    \label{fig:1D-convergence-discontinuous-f}
\end{figure}

\subsection{Fixed-resolution behavior}
The resolution condition plays a fundamental role in the numerical
simulation of highly oscillatory Helmholtz problems. In particular, it
is closely related to pollution effects and on the number of grid
points per wavelength. We therefore examine the behavior of the BPF
scheme for varying values of $k$ and $h$, and compare it with the
classical finite difference (FD) method and the dispersion-corrected
FD method~\cite{cocquet2024asymptotic}.

We consider the problem with source term
$f(x)=\sin^2(\pi x)$, together with the impedance boundary condition
\eqref{equ:nontrival-Robin-boundary-1D}. Then $f\in H^3(0,L)\cap H^2_0(0,L)$ satisfies the assumption of Theorem 4.3. The reference solution is
computed on the fine mesh $h=2^{-18}$. For the present benchmark, the
numerical data indicate that $\|u\|_{L^\infty}$ and
$\|u\|_{L^2}$ scale like $k^{-1}$, whereas the discrete energy
norm $\|u\|_{V}$ is nearly independent of $k$. We therefore use the
relative $V$-norm error as the main quantity. For
the comparison plots, we instead use the relative $L^\infty$-error,
which provides a direct pointwise measure of accuracy and gives a clear
visual comparison between the three methods.

Table~\ref{tab:convergence-wrt-kh-BPF} reports the relative
$V$-norm error of the BPF scheme for different values of $k$ and $h$.
The table contains two types of information. Along each row, the error
decreases by approximately a factor of four when $h$ is halved,
indicating a second-order convergence trend for fixed $k$. Along each
parallel diagonal, the quantity $kh$ is fixed, or equivalently, the
number of grid points per wavelength $\mathrm{PPW}=2\pi/(kh)$ is fixed. Along these diagonals, the error decreases steadily as $k$
increases, which confirms the $O(k^{-2})$ decay predicted by the
second-order convergence estimate under fixed resolution. In the
present example, the observed decay is in fact empirically faster than
this theoretical rate. The table also contains values outside the
principal Nyquist regime $0<kh<\pi$; these are still covered by the
general theory, since the main estimates hold for all
$kh\notin\pi\mathbb Z$. For example, when $k=2^{10}$ and $h=2^{-5}$, so that $kh=32$, the
relative $V$-norm error is still only $4.08\times 10^{-5}$ in this benchmark.

\begin{table}[ht]
    \centering
    \caption{Relative $V$-norm error of the BPF scheme for different values
of $k$ and $h$. Parallel diagonals correspond to fixed values of $kh$, or equivalently, fixed numbers of grid points per wavelength.}
    \begin{tabular}{||c|c|c|c|c|c|c||}
        \hline
        \diagbox{$k$}{$h$} & $2^{-5}$ & $2^{-6}$ & $2^{-7}$ & $2^{-8}$ & $2^{-9}$ & $2^{-10}$ \\
        \hline
        $2^5$ & 4.18e-05 & 1.01e-05 & 2.52e-06 & 6.27e-07 & 1.56e-07 & 3.85e-08 \\
        $2^6$ & 2.29e-05 & 5.05e-06 & 1.22e-06 & 2.96e-07 & 6.82e-08 & 1.81e-08\\
        $2^7$ & 2.48e-05 & 2.86e-06 & 6.26e-07 & 1.49e-07 & 3.41e-08 & 9.85e-09\\
        $2^8$ & 2.10e-05 & 3.00e-06 & 3.64e-07 & 8.45e-08 & 2.58e-08 & 1.31e-08\\
        $2^9$ & 6.16e-06 & 2.55e-06 & 3.76e-07 & 4.36e-08 & 8.83e-09 & 2.24e-09 \\
        $2^{10}$ & 4.08e-05 & 7.51e-07 & 3.16e-07 & 4.64e-08 & 5.08e-09 & 1.24e-09 \\
        \hline
    \end{tabular}
    \label{tab:convergence-wrt-kh-BPF}
\end{table}

When the classical finite difference (FD) method is used, increasingly fine meshes are
required for large wavenumbers because of the well-known dispersion
pollution effect. A dispersion-corrected FD method was proposed in
\cite{cocquet2024asymptotic} to alleviate this difficulty; in that
approach, the error is primarily governed by the quantity $kh$, so that
reasonable accuracy may still be obtained when the resolution is
chosen appropriately. Figure~\ref{fig:fix-kh-cpmparsion} provides a
direct visual comparison of the three methods, all evaluated at fixed values of
$kh$. It shows that the classical FD error deteriorates rapidly as $k$
increases, while the dispersion-corrected FD method improves this
behavior but still remains significantly less accurate than the BPF method over
the tested parameter range. By contrast, the BPF scheme consistently
produces the smallest errors, confirming its clear advantage in this
benchmark.
As in the relative $V$-norm data, the relative $L^\infty$-error also
shows an empirically faster decay than the $O(k^{-2})$ behavior
guaranteed by our theory, with an apparent rate close to third order
in this benchmark.

\begin{figure}[ht]
    \centering
    \includegraphics[width=0.35\linewidth]{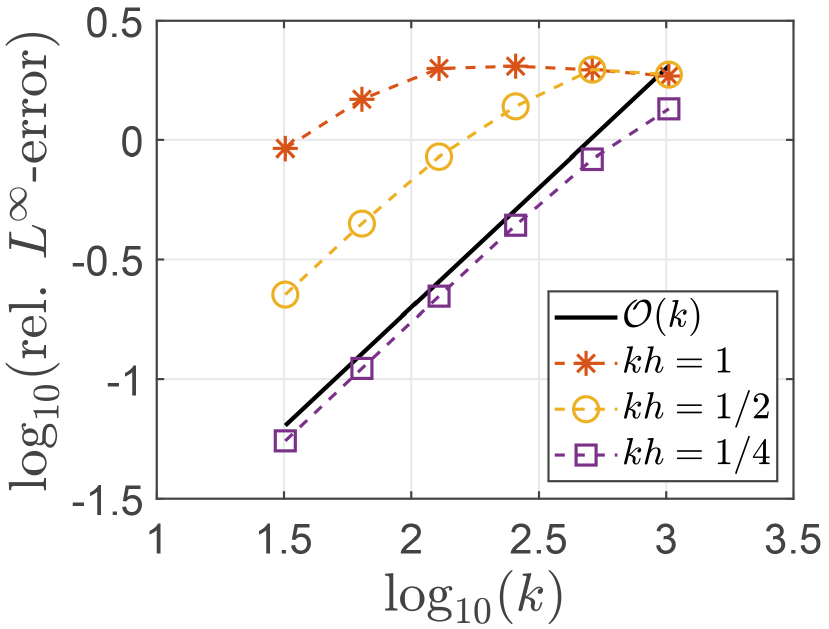}\includegraphics[width=0.33\linewidth]{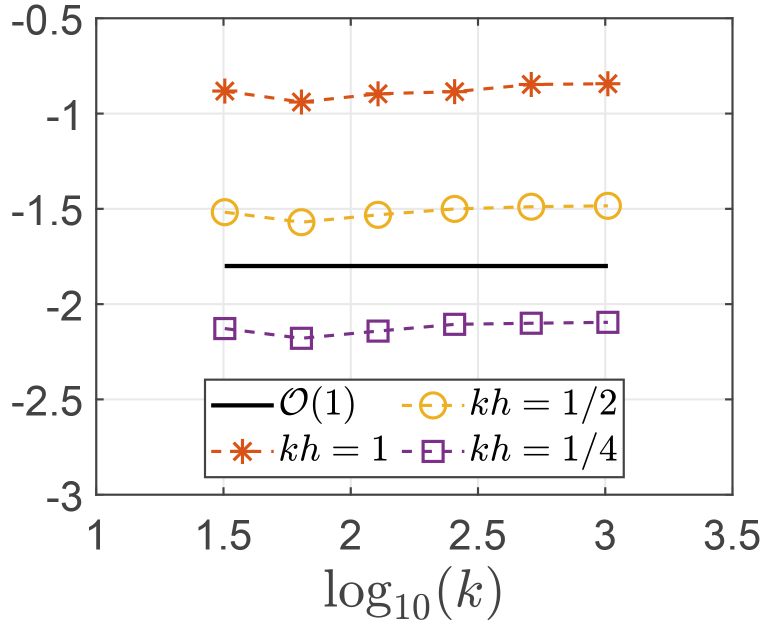}\includegraphics[width=0.32\linewidth]{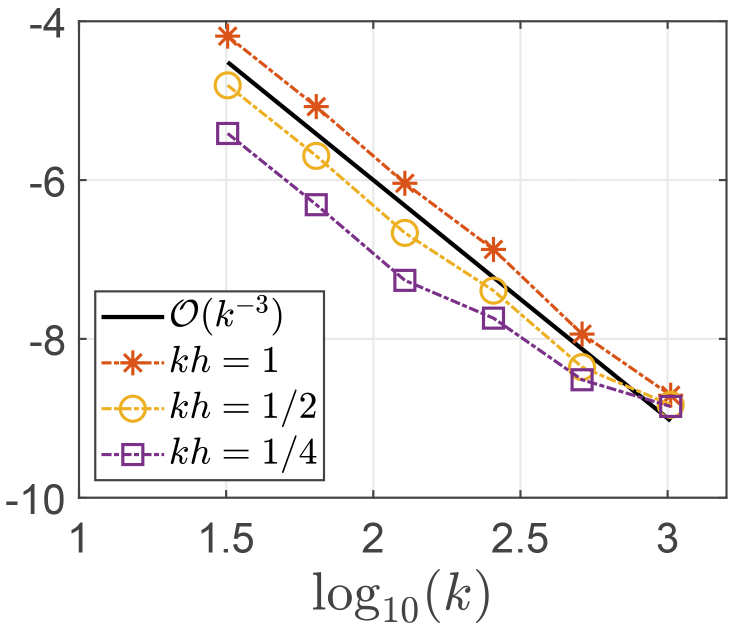}
    \caption{Relative $L^\infty$-error versus $k$ for the classical FD
method (left), dispersion-corrected FD method (middle), and BPF scheme evaluated at fixed values of $kh$ (right).}
    \label{fig:fix-kh-cpmparsion}
\end{figure}

\subsection{A preliminary 2D plane-wave test}


Although the present paper is mainly focused on the theoretical analysis of the BPF for the one-dimensional Helmholtz
equation, we briefly report a preliminary two-dimensional experiment to
illustrate the behavior of the factorized construction in a constant-coefficient plane-wave setting.
We test this discretization on
\begin{align*}
\Delta u + k^2 u = 0\quad\mbox{in }\Omega=(0,1)^2, \quad
u=f \quad\mbox{on }\pa \Omega,  
\end{align*}
where the boundary data are chosen so that the plane-wave 
\begin{equation}\label{2d plane wave}
u(x,y)=\sin\!\Big(\frac{k(x+y)}{\sqrt2}\Big)
\end{equation}
is the exact solution, as simulated in \cite{cocquet2024asymptotic}. 
We use the two-dimensional BPF scheme
\begin{equation}\label{2D BPF}
D_{x,k_1}^-D_{x,k_1}^+u_{i,j}
+
D_{y,k_2}^-D_{y,k_2}^+u_{i,j}
= f_{i,j},
\qquad 1\le i,j\le n-1,
\end{equation}
where \(D_{x,k_1}^{\pm}\) and \(D_{y,k_2}^{\pm}\) are the one-dimensional BPF
one-way difference operators applied in the \(x\)- and \(y\)-directions,
respectively.
In this test, we choose the parameter in the scheme \eqref{2D BPF} by
\[
k_1=k_2=\frac{k}{\sqrt2},
\]
which is aligned with the propagation direction of the exact plane-wave \eqref{2d plane wave}.
Thus the test is aligned with the directional construction in
Remark~\ref{remark:directional-multidimension}, and the one-dimensional plane-wave exactness mechanism applies in each coordinate direction.

\begin{table}[ht]
    \centering
    \caption{Absolute $L^\infty$-error for the preliminary 2D plane-wave test.}
    \begin{tabular}{||c|c|c|c|c|c||}
       \hline
       \diagbox{$k$}{$h$} & $\frac{1}{50}$ & $\frac{1}{100}$ & $\frac{1}{200}$ & $\frac{1}{500}$ & $\frac{1}{1000}$ \\
       \hline
       50 & 1.74e-14 & 2.51e-14 & 2.58e-14 & 3.03e-14 & 6.33e-14 \\
        200 & 1.35e-13 & 1.43e-13 & 5.80e-13 & 9.51e-14 & 1.99e-13 \\
        500 & 2.29e-13 & 3.07e-13 & 1.28e-12 & 7.82e-13 & 6.30e-13 \\
        1000 & 8.51e-13 & 4.30e-13 & 3.96e-12 & 2.72e-12 & 4.66e-12\\
        \hline
    \end{tabular}
    \label{tab:2D-plane-wave}
\end{table}

Table~\ref{tab:2D-plane-wave} reports the
absolute $L^\infty$-error for several values of $k$ and $h$. In this
constant-coefficient plane-wave benchmark, the errors remain at the
level of machine precision, indicating that the directional BPF discretization reproduces the tested two-dimensional plane-wave data
very accurately. The test serves as a numerical illustration of the directional exactness
mechanism in Remark~\ref{remark:directional-multidimension}. A full
multidimensional analysis is left for future work.

\section{Conclusion and outlook}

In this paper, we introduced a Bernoulli phase-fitted
scheme for the Helmholtz equation and developed a wavenumber-explicit analysis in the one-dimensional
impedance setting. The scheme is derived from a complexified Scharfetter--Gummel
discretization of the one-way factorization of the Helmholtz operator.
This construction preserves the factorized propagation structure and, for the homogeneous problem, reproduces sampled plane-wave solutions exactly.

For the inhomogeneous problem, we
established well-posedness, derived wavenumber-explicit stability
estimates, and proved second-order consistency and convergence. These
results hold for all $kh\notin \pi\mathbb Z$, and in
particular yield a pollution-free convergence theory for fixed
resolution within the principal Nyquist regime. 
The numerical experiments confirm the exactness
property and the predicted convergence behavior, and show favorable
fixed-resolution performance compared with standard and dispersion-corrected
finite difference methods.

The present work can be viewed as a first step in transferring the
Scharfetter--Gummel fitting principle from drift--diffusion transport to
Helmholtz wave propagation. Since the construction is local and factorized, a natural next direction is to develop multidimensional versions, in particular in finite volume or edge-based finite element frameworks. Another direction is to combine the BPF discretization
with higher-order discretizations (e.g., compact finite difference method).

\appendix
\section{Auxiliary results}

We prove some results used in this paper.



\begin{proposition}
\label{prop:factorization}
For $i\in\mathcal I_h$, the interior scheme in \eqref{discretization SG} can be written as
\begin{equation*}
D_k^- D_k^+ u_i  = \Theta(kh)\,\frac{u_{i+1}-2u_i+u_{i-1}}{h^2} + k^2 u_i,
\end{equation*}
and the boundary conditions in \eqref{discretization SG} are equivalent to
\begin{equation}\label{eq:bc-Dform}
\begin{aligned}
\frac{1}{m(kh)} D_k^+ u_0
&=
\frac{k}{\sin(kh)}
\big(u_{1}-e^{\ii kh}\,u_{0}\big), \\
\frac{1}{m(kh)} D_k^- u_n
&=
\frac{k}{\sin(kh)}
\big(e^{\ii kh}u_{n}-\,u_{n-1}\big),
\end{aligned}
\end{equation}
where $m(s)=e^{-\ii s/2}\cos(s/2)$ for $s>0$.
\end{proposition}

\begin{proof}
A direct calculation shows that 
\begin{equation}\label{transform D+D-}
    \begin{aligned}
    D_k^- D_k^+ u_i 
    &=
    \frac{1}{h}\!\left(B(\ii kh)\,D_k^-u_{i+1}-B(-\ii kh)\,D_k^-u_i\right)\\
    &=
    \frac{1}{h^2}\!\big(\Theta( kh)\,u_{i+1} 
    - \big( B^2(\ii kh) + B^2(-\ii kh) \big)u_i
    + \Theta(kh)\, u_{i-1}\big)
    \end{aligned}
\end{equation}
We deduce from \eqref{eq:basicIds} that
\begin{align*}
    B^2(\ii kh) + B^2(-\ii kh) 
    &= \big( B(\ii kh) - B(-\ii kh) \big)^2 + 2B(\ii kh) B(-\ii kh) \\
    &= - k^2h^2 + 2\Theta (kh).
\end{align*}
A substitution into \eqref{transform D+D-} leads to the interior formula.
For the boundary condition, we apply identity \eqref{eq:basicIds} again to find that
\begin{equation}\label{eq:rewrite D+}
    D_k^+u_0= \frac{1}{h}\Big(B(\ii kh)\,u_{1}-B(-\ii kh)\,u_{0}\Big)
    =
    \frac{B(\ii kh)}{h}\big(u_{1}-e^{\ii kh}\,u_{0}\big).
\end{equation}
It follows from the definition of $m(s)$ and $2\sin(s/2)\cos(s/2) = \sin(s)$ that
\begin{equation}\label{eq:key-id}
B(\ii s)
= \frac{\ii s}{e^{\ii s}-1}
= \frac{se^{-\ii s/2}}{2\sin(s/2)},
\quad\text{so}\quad
\frac{B(\ii s)}{m(s)}
=
\frac{s}{\sin(s)}.
\end{equation}
Using this identity into \eqref{eq:rewrite D+} leads to the first formula in \eqref{eq:bc-Dform}. The second formula in \eqref{eq:bc-Dform} can be derived similarly.
\end{proof}



\begin{lemma}
\label{lem:appendix-g}
Let $g(t):=\sin^2(\sqrt t)/t$ and $h(t):=\sqrt{g(t)}=\sin(\sqrt t)/\sqrt t$ for $t>0$.
Then \(g,h\in C^1(0,\infty)\) and
\begin{equation}
\label{eq:gprime-bound}
\sup_{t>0} |g'(t)| \le \frac13,
\qquad 
\sup_{t>0} |h'(t)| \le \frac16.
\end{equation}
\end{lemma}

\begin{proof}
We first estimate $h'(t)$. For \(t>0\), write \(t=s^2\) with \(s>0\). 
Then
\[
h'(t)
=
\frac{1}{2s}\,\frac{d}{ds}\Big(\frac{\sin s}{s}\Big)
=
\frac{s\cos s-\sin s}{2s^3},
\]
and the identity $\sin s-s\cos s = \int_0^s r\sin r\,dr$ yields the estimate
\[
|h'(t)|
\le
\frac{1}{2s^3}\int_0^s r|\sin r|\,dr
\le
\frac{1}{2s^3}\int_0^s r^2\,dr
=
\frac{1}{2s^3}\cdot \frac{s^3}{3}
=
\frac16.
\]
Finally, since $h(t)\le 1$, we have
$|g'(t)|\le 2\, h(t)\, |h'(t)|\le 1/3$.
This proves \eqref{eq:gprime-bound}. 
\end{proof}

\bibliographystyle{siamplain}
\bibliography{reference}

\end{document}